%% file: ex_article.tex
\documentclass[review,hidelinks,onefignum,onetabnum]{siamart250106}

\input{ex_shared}

\ifpdf
\hypersetup{
  pdftitle={An Example Article},
  pdfauthor={D. Doe, P. T. Frank, and J. E. Smith}
}
\fi

\begin{document}

\maketitle

\begin{abstract}
We consider a hybridizable discontinuous Galerkin (HDG) method for an elliptic distributed optimal control problem and we propose a balancing domain decomposition by constraints (BDDC) preconditioner to solve the discretized system. We establish an error estimate of the HDG methods with explicit tracking of a regularization parameter $\beta$. We observe that the BDDC preconditioner is robust with respect to $\beta$. Numerical results are shown to support our findings.
\end{abstract}



\section{Introduction}
In this work, we consider the following elliptic optimal control problem.
Let $\Omega$ be a bounded convex polygonal domain in $\mathbb{R}^n$ ($n=2,3$), $y_d\in \LT$ and $\beta$ be a positive constant. Find
\begin{equation}\label{optcon}
(\bar{y},\bar{u})=\argmin_{(y,u)}\left [ \frac{1}{2}\|y-y_d\|^2_{\LT}+\frac{\beta}{2}\|u\|^2_{\LT}\right],
\end{equation} 
where $(y,u)$ belongs to $H^1_0(\Omega)\times \LT$ if and only if 
\begin{equation}\label{eq:stateeq}
a(y,v)=\int_{\Omega}uv \ dx \quad \forall v\in H^1_0(\Omega).
\end{equation}
Here the bilinear form $a(\cdot,\cdot)$ is defined as
\begin{equation*}
  a(y,v)=\int_{\Omega} \nabla y\cdot \nabla v\ dx+\int_{\Omega} (\bz\cdot\nabla y) v\ dx+\int_{\Omega} \gamma yv\ dx,
\end{equation*}
where the vector field $\bz\in [W^{1,\infty}(\Omega)]^2$ and the function $\gamma\in W^{1,\infty}(\Omega)$ is nonnegative. We assume 
\begin{equation}\label{eq:advassump}
    \gamma-\frac12\nabla\cdot\bz\ge\gamma_0>0\quad a.e.\ \mbox{in}\ \Omega
\end{equation}
such that the problem \eqref{eq:stateeq} is well-posed (cf. \cite{ayuso2009discontinuous,di2011mathematical}).

It is well-known that (see \cite{Lions, Tro}) the solution of \eqref{optcon}-\eqref{eq:stateeq} is characterized by
\begin{alignat*}{3}
a(q,\bar{p})&=(\bar{y}-y_d,q)_\LT \quad &&\forall q\in H^1_0(\Omega),\\
\bar{p}+\beta\bar{u}&=0,\\
a(\bar{y},z)&=(\bar{u},z)_\LT  \quad &&\forall z\in H^1_0(\Omega),
\end{alignat*}
where $\bar{p}$ is the adjoint state.
After eliminating $\bar{u}$ (cf. \cite{hinze2005variational}), we arrive at the saddle point problem
\begin{subequations}\label{eq:osp}
\begin{alignat}{2}
a(q,\bar{p})-(\bar{y},q)_\LT&=-(y_d,q)_\LT \quad &&\forall q\in H^1_0(\Omega),\\
-(\bar{p},z)_\LT-\beta a(\bar{y},z)&=0  \quad &&\forall z \in H^1_0(\Omega).
\end{alignat}
\end{subequations}

We perform the following change of variables
\begin{equation}\label{equation:betapy}
\bar{p}=-\beta^{\frac{1}{4}}\tilde{p}
\quad 
\text{and} 
\quad 
\bar{y}=\beta^{-\frac{1}{4}}\tilde{y}
\end{equation}
so that the system \eqref{eq:osp} is more balanced with respect to $\beta$. Indeed, we have
\begin{subequations}\label{eq:sp1}
\begin{alignat}{2}
\beta^{\frac{1}{2}}a(q,\tilde{p})+(\tilde{y},q)_\LT&=\beta^{\frac{1}{4}}(y_d,q)_\LT\quad &&\forall q\in H^1_0(\Omega),\\
-(\tilde{p},z)_\LT+\beta^{\frac{1}{2}}a(\tilde{y},z)&=0 \quad &&\forall z\in H^1_0(\Omega).
\end{alignat}
\end{subequations}
We then write \eqref{eq:sp1} concisely as follows, find $(\tilde{p},\tilde{y})\in \ES$ such that
\begin{equation}\label{eq:contprob}
  \cB((\tilde{p},\tilde{y}),(q,z))=\beta^{\frac{1}{4}}(y_d,q)_\LT\quad\forall (q,z)\in\ES,
\end{equation}
where
\begin{equation}\label{eq:contdual}
  \cB((p,y),(q,z))=\beta^{\frac{1}{2}}a(q,p)+(y,q)_\LT-(p,z)_\LT+\beta^{\frac{1}{2}}a(y,z).
\end{equation}
It is well-known that \eqref{eq:sp1} has an unique solution (cf. \cite{brenner2020multigrid,liu2023robust,gaspoz2019quasi,gong2022optimal}).

\begin{remark}\label{rem:scale}
  The scaling technique \eqref{equation:betapy} is well-known (see \cite{brenner2020multigrid, gong2022optimal, gaspoz2019quasi}). However, the scaling we use here differs slightly from those in \cite{brenner2020multigrid, gong2022optimal, gaspoz2019quasi}, which is more convenient, but not essential, for the description of the BDDC algorithm in Section \ref{sec:bddc}.
\end{remark}

Numerical methods for the saddle point formulation \eqref{eq:sp1} of the optimal control problem \eqref{optcon}-\eqref{eq:stateeq} are extensively studied in the literature. For example, concrete error estimates were established in \cite{brenner2020multigrid,Scho} for $\mathbb{P}_1$ continuous Galerkin methods, similar results were also established in \cite{liu2023robust,leykekhman2012local,leykekhman2012investigation} for discontinuous Galerkin methods, and in \cite{chen2018hdg,chen2019hdg} for HDG methods. We focus on HDG methods for the optimal control problem \eqref{optcon}-\eqref{eq:stateeq} in this work. HDG methods have been intensively studied over the past two decades, see \cite{fu2015analysis,cockburn2010projection,MR2429868,cockburn2009unified} and the references therein for more details. It is well-known that traditional discontinuous Galerkin (DG) methods \cite{arnold1982interior,arnold2002unified} have more degrees of freedom than classical continuous Galerkin finite element methods when the polynomial order is less than four, see \cite{riviere2008discontinuous}. HDG overcomes this issue by using static condensation (cf. \cite{cockburn2009unified}) so that the unknowns are only on the skeleton of the mesh while maintaining the advantages of DG methods. In \cite{chen2018hdg,chen2019hdg}, the authors followed a classical approach (cf. \cite{hinze2005variational,geveci1979approximation,falk1973approximation}) to analyze the HDG methods for the optimal control problem \eqref{optcon}-\eqref{eq:stateeq}. This approach decouples the state and the adjoint state variables by introducing an intermediate problem, hence the analysis of HDG methods for single convection-diffusion PDEs (cf. \cite{fu2015analysis}) can be utilized. However, in \cite{chen2018hdg}, the regularization parameter $\beta$ was not taken into account and only $L_2$ error estimates are provided.

The BDDC algorithm is a widely used non-overlapping domain decomposition method. 
Initially introduced in \cite{Doh04} for symmetric positive definite problems, it has since been extended to solve non-symmetric positive definite systems \cite{TuLi:2008:BDDCAD, TZAD2021}. BDDC algorithms have also been extended to solve saddle point problems \cite{TWZstokes2020,LiBDDCS,TZOseen2022}, where the original problems can be reduced to symmetric positive definite problems or non-symmetric positive definite problem through a benign space approach. The BDDC algorithm is usually used as a preconditioner for CG and GMRES algorithms.

Our contribution in this work is two-fold. First, we extend the framework in \cite{brenner2020multigrid,liu2023robust,gong2022optimal} to HDG methods. As mentioned above, the analyses in \cite{chen2018hdg,chen2019hdg,leykekhman2012local,leykekhman2012investigation} utilized an intermediate problem to decouple the state and the adjoint state variables while an inf-sup condition was used in \cite{brenner2020multigrid,liu2023robust,Scho} where the state and adjoint state variables are considered simultaneously. The advantages of using the approaches in \cite{brenner2020multigrid,liu2023robust} are the following:
\begin{itemize}
  \item The convergence is established in a natural energy norm that leads to concrete error estimates with shorter and more elegant proofs.
  \item By using a change of variable mentioned in Remark \ref{rem:scale}, it is more convenient to track the parameter $\beta$ during the analysis, which is an important parameter in the optimal control problems. 
  \item These energy estimates often are used to prove robustness of the corresponding fast solvers (cf. \cite{brenner2020multigrid,liu2023robust,liu2024multigrid}) which are theoretically faster since they converge in an energy norm.
  \item This approach can also be easily extended to the case of convection-dominated state equations \cite{liu2024multigrid}. 
\end{itemize}  
Secondly, we propose a BDDC algorithm to solve the discretized system. This is an extension of the work in \cite{TWZstokes2020,TZAD2021,TuLi:2008:BDDCAD}. We observe that the BDDC algorithm is robust with respect to the parameter $\beta$. This phenomenon is consistent with the work in \cite{brenner2020multigrid,liu2024multigrid,liu2023robust} for multigrid methods. Similar to the results in \cite{LiBDDCS,TWZstokes2020,TuLi:2008:BDDCAD,TZAD2021,TZOseen2022,ZT:2016:Darcy}, the convergence iterations of BDDC algorithms are scalable and independent with increasing number of subdomains. The detailed convergence analysis of BDDC algorithms for symmetric positive definite problems has been given in \cite{LiBDDCS,TWZstokes2020,ZT:2016:Darcy}.  For non-symmetric or indefinite problems, the upper bound and lower bound estimates are established in \cite{TuLi:2008:BDDCAD,TZAD2021,TZOseen2022}.

The rest of the paper is organized as follows. In Section \ref{sec:hdg}, we discuss the HDG formulation of \eqref{eq:contprob} and present some preliminary estimates useful for the analysis. In Section \ref{sec:convhdg}, we derive a concrete error estimate for the HDG methods using an inf-sup condition, along with suitable assumptions on the stabilizers. The parameter $\beta$ is explicitly tracked. We then introduce a BDDC preconditioner in Section \ref{sec:bddc} to solve the discretized system and present numerical results in Section \ref{sec:numerics}. Finally, we end with some concluding remarks in Section \ref{sec:conclude}.

Throughout this paper, we use $C$
 (with or without subscripts) to denote a generic positive constant that is independent of any mesh parameter and $\beta$, unless otherwise stated. In addition, to avoid the proliferation of constants, we use the notation $A\lesssim B$ (or $A\gtrsim B$) to represent $A\leq \text{(constant)}B$. The notation $A\approx B$ is equivalent to $A\lesssim B$ and $B\lesssim A$.

\section{HDG Discretization and Preliminary estimates}\label{sec:hdg}
In this section, we discuss the HDG discretization for the optimal control problem and give some preliminary estimates that are useful for the analysis. For generality, let us consider a more general problem. Find $(p,y)\in \ES$ such that
\begin{equation}\label{eq:gsaddle}
  \cB((p,y),(q,z))=(f,q)_\LT+(g,z)_\LT\quad\forall (q,z)\in\ES,
\end{equation}
where $f$ and $g$ are sufficiently smooth and $\cB$ is defined in \eqref{eq:contdual}.
 Note that $\eqref{eq:gsaddle}$ is equivalent to the following equations,
\begin{subequations}\label{Osystem}
 \begin{alignat}{3}
 \beta^{\frac12}(-\Delta p{-}\nabla\cdot(\bz p)+\gamma p)+y&=f\quad&&\mbox{in}\quad\Omega,\\
 p&=0\qquad &&\mbox{on}\quad\partial\Omega,\\
 \beta^{\frac12}(-\Delta y{+}\bz\cdot\nabla{y}+\gamma y){-}p&=g\quad&&\mbox{in}\quad\Omega,\\
 y&=0\quad&&\mbox{on}\quad\partial\Omega.
 \end{alignat}
 \end{subequations}
 Let $\bq=-\nabla y$ and $\bp=-\nabla p$. we can also write \eqref{Osystem} as a first-order system as follows,
\begin{subequations}\label{eq:fsystem}
 \begin{alignat}{3}
 \bp+\nabla p&=0\quad&&\mbox{in}\quad\Omega,\\
 \beta^{\frac12}(\nabla\cdot \bp-\nabla\cdot(\bz p)+\gamma p)+y&=f\quad&&\mbox{in}\quad\Omega,\\
 p&=0\qquad &&\mbox{on}\quad\partial\Omega,\\
\bq+\nabla y&=0\quad&&\mbox{in}\quad\Omega,\\
 \beta^{\frac12}(\nabla\cdot\bq+\bz\cdot\nabla y+\gamma y){-}p&=g\quad&&\mbox{in}\quad\Omega,\\
 y&=0\quad&&\mbox{on}\quad\partial\Omega.
 \end{alignat}
 \end{subequations}

 \begin{remark}[Regularity]\label{rem:regularity}
   Throughout the paper, we assume the solutions to \eqref{eq:fsystem} are sufficiently smooth. This is reasonable since we only consider convex polygonal domains and sufficiently smooth right-hand sides (see \cite{Gris}). Note that if we only assume $(f,g)\in \LT\times\LT$, the following regularity is valid on convex domains (see \cite{brenner2020multigrid}),
   \begin{equation*}
      \|\beta^{\frac{1}{2}}p\|_{H^2(\O)}+\|\beta^{\frac{1}{2}}y\|_{H^2(\O)}\leq
  C_{\O}(\|f\|_\LT+\|g\|_\LT),
   \end{equation*}
   where $(p,y)$ are the solutions to \eqref{Osystem}.
 \end{remark}

Our first goal is to solve \eqref{eq:fsystem} using HDG discretization and establish the corresponding error estimates (cf. \cite{chen2018hdg}). 
\subsection{HDG formulation}
Let $\mathcal{T}_h$ be a quasi-uniform, shape regular simplicial triangulation of $\Omega$.
The diameter of $K\in\mathcal{T}_h$ is denoted by $h_K$ and $h=\max_{K\in\mathcal{T}_h}h_K$ is the mesh diameter. 
Let $\mathcal{E}_h=\mathcal{E}^b_h\cup\mathcal{E}^i_h$ where $\cE^i_h$ (resp., $\cE^b_h$) represents the set of interior edges (resp., boundary edges).
Define the discrete spaces as follows:
\begin{align*}
&\bm{V}_h=\{\bm{v}\in (L^2(\Omega))^n:\bm{v}|_K\in (\mathbb{P}^k(K))^n,\ \forall K\in\mathcal{T}_h\},\\
&W_h=\{w\in L^2(\Omega): w|_K\in \mathbb{P}^k(K),\ \forall K\in\mathcal{T}_h\},\\
&\Lambda_h=\{\mu\in L^2(\mathcal{E}_h):\mu|_{e}\in \mathbb{P}^k(e),\ \forall e\in \mathcal{E}_h \},\\
&\Lambda^0_h=\{\mu\in \Lambda_h:\mu|_{e}=0,\ \forall e\in\partial\Omega\}.
\end{align*}
Here $k$ is a nonnegative integer and $e$ represents an edge or face in the triangulation $\cT_h$.
We also denote
\begin{equation*}
  \begin{aligned}
    (\eta,\xi)_{\cT_h}=\sum_{K\in\cT_h}\int_K \eta \xi\ \!dx,\quad \l\eta,\xi\r_{\partial\cT_h}=\sum_{K\in\cT_h}\int_{\partial K} \eta \xi\ \!ds.
  \end{aligned}
\end{equation*}
See Figure \ref{fig:hdg} for an illustration of the HDG degree of freedoms. Note that one of the most important features of HDG is that it requires solving only for the degrees of freedom on the edges, while the solution within each triangle can be recovered using these edge degrees of freedom.

\begin{figure}[!h]
\centering
\begin{tikzpicture}
\draw[thick] plot coordinates {(0,1) (0,-1)};
\draw[thick] plot coordinates {(-0.2,1) (-0.2,-1)};
\draw[thick] plot coordinates {(0.2,1) (0.2,-1)};
\draw[thick] plot coordinates {(-1.932,0) (-0.2,1)};
\draw[thick] plot coordinates {(-1.932,0) (-0.2,-1)};
\draw[thick] plot coordinates {(1.932,0) (0.2,1)};
\draw[thick] plot coordinates {(1.932,0) (0.2,-1)};

\node at (0,1) [circle,fill,inner sep=1.5pt]{};
\node at (0,-1) [circle,fill,inner sep=1.5pt]{};


\node at (0.2,-1) [circle,fill,inner sep=1.5pt]{};
\node at (0.2,1) [circle,fill,inner sep=1.5pt]{};
\node at (-0.2,-1) [circle,fill,inner sep=1.5pt]{};
\node at (-0.2,1) [circle,fill,inner sep=1.5pt]{};
\node at (-1.932,0) [circle,fill,inner sep=1.5pt]{};
\node at (1.932,0) [circle,fill,inner sep=1.5pt]{};

\draw[thick] plot coordinates {(5,1) (5,-1)};
\draw[thick] plot coordinates {(4.8,1) (4.8,-1)};
\draw[thick] plot coordinates {(5.2,1) (5.2,-1)};
\draw[thick] plot coordinates {(3.068,0) (4.8,1)};
\draw[thick] plot coordinates {(3.068,0) (4.8,-1)};
\draw[thick] plot coordinates {(6.932,0) (5.2,1)};
\draw[thick] plot coordinates {(6.932,0) (5.2,-1)};

\node at (5,1) [circle,fill,inner sep=1.5pt]{};
\node at (5,0) [circle,fill,inner sep=1.5pt]{};
\node at (5,-1) [circle,fill,inner sep=1.5pt]{};


\node at (5.2,-1) [circle,fill,inner sep=1.5pt]{};
\node at (5.2,0) [circle,fill,inner sep=1.5pt]{};
\node at (5.2,1) [circle,fill,inner sep=1.5pt]{};
\node at (4.8,-1) [circle,fill,inner sep=1.5pt]{};
\node at (4.8,0) [circle,fill,inner sep=1.5pt]{};
\node at (4.8,1) [circle,fill,inner sep=1.5pt]{};
\node at (3.068,0) [circle,fill,inner sep=1.5pt]{};
\node at (6.932,0) [circle,fill,inner sep=1.5pt]{};

\node at (3.934,0.5) [circle,fill,inner sep=1.5pt]{};
\node at (6.066,0.5) [circle,fill,inner sep=1.5pt]{};
\node at (3.934,-0.5) [circle,fill,inner sep=1.5pt]{};
\node at (6.066,-0.5) [circle,fill,inner sep=1.5pt]{};
\end{tikzpicture}
\caption{HDG degrees of freedom for $k=1$ and $k=2$}
\label{fig:hdg}
\end{figure}
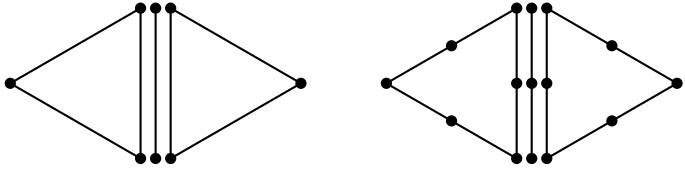

The HDG method for \eqref{eq:fsystem} is to find 
 $(\bm{q}_h,\bm{p}_h, y_h, p_h, \widehat{y}_h,\widehat{p}_h) \in \bm{V}_h\times\bm{V}_h\times W_h\times W_h\times\Lambda^0_h\times\Lambda^0_h $ such that,
\begin{subequations}\label{eq:hdg1}
  \begin{alignat}{1}
(\bm{q}_h,\bm{r}_1)_{\cT_h}-( y_h,\nabla\cdot\bm{r}_1)_{\cT_h}+\l\widehat{{y}}_h,\bm{r}_1\cdot\bm{n}\r_{\partial\cT_h}&=0,\\
(\bm{p}_h,\bm{r}_2)_{\cT_h}-( p_h,\nabla\cdot\bm{r}_2)_{\cT_h}+\l\widehat{{p}}_h,\bm{r}_2\cdot\bm{n}\r_{\partial\cT_h}&=0,\\
-\beta^{\frac12}(\bm{q}_h+\bz y_h,\nabla w_1)_{\cT_h}-( p_h,w_1)_{\cT_h}+\beta^{\frac12}((\gamma-\nabla\cdot\bm{\zeta})y_h, w_1)_{\cT_h}&\\
+\beta^{\frac12}\l\widehat{\bm{q}_h}\cdot\bm{n}+\bz\cdot\bm{n}\widehat{y}_h,w_1\r_{\partial\cT_h}&=(g,w_1)_{\cT_h},\nonumber\\
-\beta^{\frac12}(\bm{p}_h-\bz p_h,\nabla w_2)_{\cT_h}+( y_h,w_2)_{\cT_h}+\beta^{\frac12}(\gamma p_h, w_2)_{\cT_h}&\\
+\beta^{\frac12}\l\widehat{\bm{p}_h}\cdot\bm{n}-\bz\cdot\bm{n}\widehat{p}_h,w_2\r_{\partial\cT_h}
&=(f,w_2)_{\cT_h},\nonumber\\
-\l\widehat{\bm{q}_h}\cdot\bm{n}+\bz\cdot\bm{n}\widehat{y}_h,\mu_1\r_{\partial\cT_h}&=0,\\
-\l\widehat{\bm{p}_h}\cdot\bm{n}-\bz\cdot\bm{n}\widehat{p}_h,\mu_2\r_{\partial\cT_h}&=0,
\end{alignat}
\end{subequations}
for all $(\bm{r}_1,\bm{r}_2,w_1,w_2,\mu_1,\mu_2)\in\bm{V}_h\times\bm{V}_h\times W_h\times W_h\times\Lambda^0_h\times\Lambda^0_h$,
where the numerical fluxes are defined as
\begin{align*}
\widehat{\bm{q}_h}\cdot\bm{n}=\bm{q}_h\cdot\bm{n}+\tau_1( y_h-\widehat{y}_h),\quad \widehat{\bm{p}_h}\cdot\bm{n}=\bm{p}_h\cdot\bm{n}+\tau_2( p_h-\widehat{p}_h).
\end{align*}
Here the stabilizers $\tau_1$ and $\tau_2$ will be discussed in Section \ref{ssec:assumptau}. Note that for each element $K\in\cT_h$, the local problem of HDG method is satisfied such that 
\begin{subequations}\label{eq:hdg1local}
  \begin{alignat}{1}
  (\bm{q}_h,\bm{r}_1)_{K}-( y_h,\nabla\cdot\bm{r}_1)_{K}+\l\widehat{{y}}_h,\bm{r}_1\cdot\bm{n}\r_{K}&=0,\label{eq:qK}\\
(\bm{p}_h,\bm{r}_2)_{K}-( p_h,\nabla\cdot\bm{r}_2)_{K}+\l\widehat{{p}}_h,\bm{r}_2\cdot\bm{n}\r_{K}&=0,\label{eq:pK}\\
-\beta^{\frac12}(\bm{q}_h+\bz y_h,\nabla w_1)_{K}-( p_h,w_1)_{K}+\beta^{\frac12}((\gamma-\nabla\cdot\bm{\zeta})y_h, w_1)_{K}&\label{eq:yK}\\
+\beta^{\frac12}\l\widehat{\bm{q}_h}\cdot\bm{n}+\bz\cdot\bm{n}\widehat{y}_h,w_1\r_{\partial K}&=(g,w_1)_{K},\nonumber\\
-\beta^{\frac12}(\bm{p}_h-\bz p_h,\nabla w_2)_{K}+( y_h,w_2)_{K}+\beta^{\frac12}(\gamma p_h, w_2)_{K}&\label{eq:pKu}\\
+\beta^{\frac12}\l\widehat{\bm{p}_h}\cdot\bm{n}-\bz\cdot\bm{n}\widehat{p}_h,w_2\r_{\partial K}
&=(f,w_2)_{K}\nonumber,
  \end{alignat}
  \end{subequations}
for any $({\bm r}_1, {\bm r}_2,w_1, w_2)\in \bm{V}_h(K)\times\bm{V}_h(K)\times W_h(K)\times W_h(K)$ (cf. \cite{fu2015analysis,TZAD2021}).

Moreover, the equations \eqref{eq:hdg1} are equivalent to find $(\bm{q}_h,\bm{p}_h, y_h, p_h, \widehat{y}_h,\widehat{p}_h) \in \bm{V}_h\times\bm{V}_h\times W_h\times W_h\times\Lambda^0_h\times\Lambda^0_h $ such that,
\begin{subequations}\label{eq:hdg2}
  \begin{alignat}{1}
\beta^{\frac12}(\bm{q}_h,\bm{r}_1)_{\cT_h}-\beta^{\frac12}( y_h,\nabla\cdot\bm{r}_1)_{\cT_h}+\beta^{\frac12}\l\widehat{{y}}_h,\bm{r}_1\cdot\bm{n}\r_{\partial\cT_h}&=0,\label{eq:qh1}\\
\beta^{\frac12}(\bm{p}_h,\bm{r}_2)_{\cT_h}-\beta^{\frac12}( p_h,\nabla\cdot\bm{r}_2)_{\cT_h}+\beta^{\frac12}\l\widehat{{p}}_h,\bm{r}_2\cdot\bm{n}\r_{\partial\cT_h}&=0,\label{eq:ph1}\\
\beta^{\frac12}(\nabla\cdot\bm{q}_h,w_1)_{\cT_h}-\beta^{\frac12}( y_h,\bz\cdot\nabla{w}_1)_{\cT_h}
+\beta^{\frac12}((\gamma-\nabla\cdot\bm{\zeta})y_h, w_1)_{\cT_h}\\\nonumber
-( p_h,w_1)_{\cT_h}+\beta^{\frac12}\l\tau_1 y_h,{ w}_1\r_{\partial \cT_h}
+\beta^{\frac12}\l(\bz\cdot\bm{n}-\tau_1)\widehat{y}_h,w_1\r_{\partial \cT_h}&=(g,w_1)_{\cT_h},\\
\beta^{\frac12}(\nabla\cdot\bm{p}_h,w_2)_{\cT_h}+\beta^{\frac12}( p_h,\bz\cdot\nabla w_2)_{\cT_h}
+\beta^{\frac12}(\gamma p_h, w_2)_{\cT_h}\\\nonumber
+(y_h,w_2)_{\cT_h}+\beta^{\frac12}\l\tau_2 p_h,w_2\r_{\partial\cT_h}-\beta^{\frac12}\l(\tau_2+\bz\cdot\bm{n})\widehat{p}_h,w_2)\r_{\partial\cT_h}&=(f,w_2)_{\cT_h},\\
-\beta^{\frac12}\l\bm{q}_h\cdot\bm{n},\mu_1\r_{\partial \cT_h}-\beta^{\frac12}\l\tau_1 y_h,\mu_1\r_{\partial \cT_h}-\beta^{\frac12}\l(\bz\cdot\bm{n}-\tau_1)\widehat{y}_h,\mu_1)&\r_{\partial\cT_h}=0,\\
-\beta^{\frac12}\l\bm{p}_h\cdot\bm{n},\mu_2\r_{\partial \cT_h}-\beta^{\frac12}\l\tau_2 p_h,\mu_2\r_{\partial \cT_h}+\beta^{\frac12}\l(\bz\cdot\bm{n}+\tau_2)\widehat{p}_h,\mu_2)&\r_{\partial\cT_h}=0,
\end{alignat}
\end{subequations}
for all $(\bm{r}_1,\bm{r}_2,w_1,w_2,\mu_1,\mu_2)\in\bm{V}_h\times\bm{V}_h\times W_h\times W_h\times\Lambda^0_h\times\Lambda^0_h$. The system \eqref{eq:hdg2} is more suitable for the description of the BDDC algorithm.

\subsection{Concise form}

We can write the HDG methods \eqref{eq:hdg2} concisely as follows (cf. \cite{lehrenfeld2010hybrid}). Find $((\bm{q}_h, y_h, \widehat{y}_h),(\bm{p}_h, p_h, \widehat{p}_h))\in (\bm{V}_h\times W_h\times\Lambda^0_h)\times(\bm{V}_h\times W_h\times\Lambda^0_h)$ such that
\begin{equation}\label{eq:hdgconcise}
\begin{aligned}
  &\cB_h(((\bm{q}_h, y_h, \widehat{y}_h),(\bm{p}_h, p_h, \widehat{p}_h)),((\bm{r}_1,w_1,\mu_1),(\bm{r}_2,w_2,\mu_2)))\\
&=(g,w_1)_{\cT_h}+(f,w_2)_{\cT_h},
\end{aligned}
\end{equation}
for all $((\bm{r}_1,w_1,\mu_1),(\bm{r}_2,w_2,\mu_2))\in (\bm{V}_h\times W_h\times\Lambda^0_h)\times(\bm{V}_h\times W_h\times\Lambda^0_h)$,
where $f$ and $g$ are sufficiently smooth and the bilinear form $\cB_h$ is defined as,
\begin{equation}
\begin{aligned}
  &\cB_h(((\bm{q}_h, y_h, \widehat{y}_h),(\bm{p}_h, p_h, \widehat{p}_h)),((\bm{r}_1,w_1,\mu_1),(\bm{r}_2,w_2,\mu_2)))\\
  =&\beta^\frac12 a_{2,h}((\bm{p}_h, p_h, \widehat{p}_h),(\bm{r}_2,w_2,\mu_2))+( y_h,w_2)_{\cT_h}\\
  &-( p_h,w_1)_{\cT_h}+\beta^\frac12a_{1,h}((\bm{q}_h, y_h, \widehat{y}_h),(\bm{r}_1,w_1,\mu_1)).
\end{aligned}
\end{equation}
Here the bilinear forms $a_{1,h}$ and $a_{2,h}$ are defined as 
\begin{equation}\label{eq:hdgah1}
\begin{aligned}
    a_{1,h}&((\bm{q}, v, \lambda),(\bm{r},w,\mu))\\
    =&(\bq,\br)_{\cT_h}-(v,\nabla\cdot\br)_{\cT_h}+\l \lambda,\br\cdot\bn\r_{\partial\cT_h}\\
    &+(\nabla\cdot\bq ,w)_{\cT_h}-(v,\bm{\zeta}\cdot \nabla w)_{\cT_h}+((\gamma-\nabla\cdot\bm{\zeta})v, w)_{\cT_h}\\
  &+\l\lambda\bz\cdot\bn +\tau_1(v-\lambda),w\r_{\partial\cT_h}-\l(\bq+\lambda\bz)\cdot\bn +\tau_1(v-\lambda),\mu\r_{\partial\cT_h}
\end{aligned}
\end{equation}
and
\begin{equation}\label{eq:hdgah2}
\begin{aligned}
    a_{2,h}&((\bm{q}, v, \lambda),(\bm{r},w,\mu))\\
    =&(\bq,\br)_{\cT_h}-(v,\nabla\cdot\br)_{\cT_h}+\l \lambda,\br\cdot\bn\r_{\partial\cT_h}\\
    &+(\nabla\cdot\bq ,w)_{\cT_h}+(v,\bm{\zeta}\cdot \nabla w)_{\cT_h}+(\gamma v, w)_{\cT_h}\\
  &-\l\lambda\bz\cdot\bn +\tau_2(\lambda-v),w\r_{\partial\cT_h}-\l(\bq-\lambda\bz)\cdot\bn +\tau_2(v-\lambda),\mu\r_{\partial\cT_h}.
\end{aligned}
\end{equation}

\subsection{Assumptions on the stabilizers $\tau_1$ and $\tau_2$}\label{ssec:assumptau}

We state the following assumptions about the stabilizers $\tau_1$ and $\tau_2$ (cf. \cite{fu2015analysis,chen2018hdg,TZAD2021}). 
\begin{assumption}\label{assump:tau}
For the stabilizers $\tau_1$ and $\tau_2$, we have the following  assumptions:
\begin{enumerate}[label=(\theassumption\alph*), ref=\theassumption\alph*]
\item\label{assump1} $\tau_1$ is a piecewise positive constant on $\partial\cT_h$ and there exists a constant $C_1$ such that $\tau_1\le C_1$.
\item\label{assump2} $\tau_1=\tau_2+\bm{\zeta}\cdot\bm{n}$.
\item\label{assump3} $\inf\limits_{x\in e}(\tau_1-\frac12\bm{\zeta}\cdot\bm{n})\geq C^*\max_{x\in e}|\bm{\zeta}(x)\cdot\bm{n}|,\ \forall e\in\partial K$ and $\forall K\in\cT_h.$
\end{enumerate}
\end{assumption}

\subsection{Projection operators}

We introduce several standard projection operators that are needed in the error analysis.  Let $\Pi_h :H^1(\cT_h)\rightarrow W_h$ be the projection operator defined as follows 

  \begin{alignat}{3}\label{eq:hdgproj}
    (\Pi_hv, w)_K&=(v,w)_K\quad&&\forall w\in \mathbb{P}^{k}(K),
  \end{alignat}
and $P_\Lambda: L^2(\cE_h)\rightarrow \Lambda_h$ be the $L^2$ orthogonal projection defined as
\begin{equation}\label{eq:hdgprojor}
  \l P_\Lambda v,\mu\r_e=\l v, \mu\r_e\quad\forall e\in \cE_h\ \text{and}\ \mu\in \Lambda_h.
\end{equation}

Let $\bm{\Pi}_h: (H^1(\cT_h))^n\rightarrow \bm{V}_h$
be defined as $\bm{\Pi}_h\bm{v}:=(\Pi_h v_1,\Pi_h v_2,\ldots, \Pi_h v_n)$. Then it is trivial to check that $(\bm{\Pi}_h\bm{v},\bm{w})_K=(\bm{v},\bm{w})_K$ for all $\bm{w}$ in $(\mathbb{P}^{k}(K))^n$.
Let $({\bq},{\bp},{y},{p})$ be sufficiently smooth and let $(\bq_h,\bp_h,y_h,p_h,\widehat{y},\widehat{p}_h)$ belong to $\bm{V}_h\times\bm{V}_h\times W_h\times W_h\times\Lambda^0_h\times\Lambda^0_h $.
For convenience, we define the following terms:
\begin{equation}\label{eq:projnotations}
  \begin{alignedat}{4}
&\epsilon^{\bq}_h={\bq}_h-\bm{\Pi}_h{\bq},\quad &&\delta^{\bq}_h={\bq}-\bm{\Pi}_h{\bq},\quad&&\epsilon^{y}_h=y_h-\Pi_h{y}, \quad &&\delta^{y}_h=y-\Pi_hy,\\
&\epsilon^{\bp}_h=\bp_h-\bm{\Pi}_h{\bp},\quad &&\delta^{\bp}_h={\bp}-\bm{\Pi}_h{\bp},\quad&&\epsilon^{p}_h=p_h-\Pi_h{p}, \quad &&\delta^{p}_h=p-\Pi_hp,\\
&\epsilon^{\widehat{y}}_h=\widehat{y}_h-P_\Lambda y,\quad &&\delta^{\widehat{y}}_h=y|_{e}-P_\Lambda y,\quad&&\epsilon^{\widehat{p}}_h=\widehat{p}_h-P_\Lambda p,\quad &&\delta^{\widehat{p}}_h=p|_{e}-P_\Lambda p.
\end{alignedat}
\end{equation}
We also define the norms
\begin{equation*}
  \|\eta\|^2_{\cT_h}=(\eta,\eta)_{\cT_h},\quad \|\eta\|^2_{1,\cT_h}=(\nabla\eta,\nabla\eta)_{\cT_h},\quad\|\eta\|^2_{\partial\cT_h}=(\eta,\eta)_{\partial\cT_h}.
\end{equation*}
Moreover, we use the notation $|\cdot|_{k+1}:=|\cdot|_{H^{k+1}(\Omega)}$ for simplicity.
It is known that (cf. \cite{BS,lehrenfeld2010hybrid}) the following estimates hold:
\begin{equation}\label{eq:projection}
\begin{alignedat}{3}
&\|\delta^{\bm{q}}_h\|_{\cT_h}+h\|\delta^{\bm{q}}_h\|_{1,\cT_h}
\leq Ch^{k+1}|{\bm{q}}|_{k+1},\quad&&\|\delta^{y}_h\|_{\cT_h}+h\|\delta^{y}_h\|_{1,\cT_h}
\leq Ch^{k+1}|y|_{k+1},\\
&\|\delta^{\bm{p}}\|_{\cT_h}+h\|\delta^{\bm{p}}_h\|_{1,\cT_h}\leq Ch^{k+1}|{\bm{p}}|_{k+1},\quad &&\|\delta^{p}_h\|_{\cT_h}+h\|\delta^{p}_h\|_{1,\cT_h}
\leq Ch^{k+1}|p|_{k+1},\\
&\|\delta^{\widehat{y}}_h\|_{\partial\cT_h}\leq Ch^{k+\frac12}|y|_{k+1},\quad&&\|\delta^{\widehat{p}}_h\|_{\partial\cT_h}\leq Ch^{k+\frac12}|p|_{k+1}.
\end{alignedat}
\end{equation}


\begin{remark}
  Note that we use standard projection operators in our subsequent analysis (cf. \cite{lehrenfeld2010hybrid}). More sophisticated HDG projection operators were established and utilized in \cite{MR2429868,cockburn2010projection,chen2018hdg} to obtain superconvergent results.
\end{remark}

\section{Convergence analysis for HDG methods}\label{sec:convhdg}

In this section, we establish concrete estimates for the HDG methods \eqref{eq:hdg2} in an energy norm, with $\beta$ explicitly tracked. We first establish some important properties of the bilinear forms $a_{1,h}$ and $a_{2,h}$.

\subsection{Properties of $a_{1,h}$ and $a_{2,h}$}

\begin{lemma}
  Under Assumption \ref{assump:tau}, we have, for any $((\bm{q}_h, y_h, \widehat{y}_h),(\bm{p}_h, p_h, \widehat{p}_h))\in (\bm{V}_h\times W_h\times\Lambda^0_h)\times(\bm{V}_h\times W_h\times\Lambda^0_h)$ and $((\bm{r}_1,w_1,\mu_1),(\bm{r}_2,w_2,\mu_2))\in (\bm{V}_h\times W_h\times\Lambda^0_h)\times(\bm{V}_h\times W_h\times\Lambda^0_h)$,
  \begin{equation}\label{eq:hdgahduala}
  a_{2,h}((\bm{p}_h, p_h, \widehat{p}_h),(\bm{r}_2,w_2,\mu_2))=a_{1,h}((-\bm{r}_2,w_2,\mu_2),(-\bm{p}_h, p_h, \widehat{p}_h)).
\end{equation}
\end{lemma}
\begin{proof}
  First we notice that, under the assumption \eqref{assump2}, we have
  \begin{equation}\label{eq:hdgahdual}
\begin{aligned}
  &a_{1,h}((-\bm{r},w,\mu),(-\bm{q}, v, \lambda))\\
  =&(\bq,\br)_{\cT_h}+(w,\nabla\cdot\bq)_{\cT_h}-\l \mu,\bq\cdot\bn\r_{\partial\cT_h}\\
    &-(\nabla\cdot\br ,v)_{\cT_h}-(w,\bm{\zeta}\cdot \nabla v)_{\cT_h}+((\gamma-\nabla\cdot\bm{\zeta})w, v)_{\cT_h}\\
  &+\l\mu\bz\cdot\bn +\tau_1(w-\mu),v\r_{\partial\cT_h}-\l(-\br+\mu\bz)\cdot\bn +\tau_1(w-\mu),\lambda\r_{\partial\cT_h}\\
  =&(\bq,\br)_{\cT_h}-(v,\nabla\cdot\br)_{\cT_h}+\l \lambda,\br\cdot\bn\r_{\partial\cT_h}\\
  &+(w,\nabla\cdot\bq)_{\cT_h}+(v,\bm{\zeta}\cdot\nabla w)_{\cT_h}+(\gamma w, v)_{\cT_h}\\
  &\l\lambda\bm{\zeta}\cdot\bn+\tau_2(\lambda-v), w\r_{\partial\cT_h}-\l\bq\cdot\bn+\tau_2(v-\lambda),\mu\r_{\partial\cT_h}.
\end{aligned}
\end{equation}
Compare $a_{2,h}((\bm{p}_h, p_h, \widehat{p}_h),(\bm{r}_2,w_2,\mu_2))$ with $a_{1,h}((-\bm{r}_2,w_2,\mu_2),(-\bm{p}_h, p_h, \widehat{p}_h))$ and use \eqref{eq:hdgahdual}, notice that the only difference is the term $\l\widehat{p}_h\bm{\zeta}\cdot\bn,\mu_2\r_{\partial\cT_h}$. However, we have $\l\widehat{p}_h\bm{\zeta}\cdot\bn,\mu_2\r_{\partial\cT_h}=0$ since $\widehat{p}_h$ is single valued on the interior faces and $\widehat{p}_h=0$ on the boundary $\partial\O$. The relation \eqref{eq:hdgahduala} then follows.
\end{proof}

\begin{remark}
The relation \eqref{eq:hdgahduala} indicates that the discrete bilinear forms $a_{1,h}$ and $a_{2,h}$ are almost dual to each other, slightly different from the continuous problem \eqref{eq:contdual}. Similar results can be found in \cite{chen2019hdg,chen2018hdg}.
\end{remark}

\begin{lemma}\label{lem:ahnorm}
  For any $(\br_h, w_h, \mu_h)\in (\bm{V}_h\times W_h\times\Lambda^0_h)$, we have
  \begin{alignat*}{2}
    &a_{1,h}((\br_h, w_h, \mu_h),(\br_h, w_h, \mu_h))\\
    &=(\br_h,\br_h)_{\cT_h}+((\gamma-\frac12\nabla\cdot\bm{\zeta})w_h, w_h)_{\cT_h}+\l(\tau_1-\frac12\bz\cdot\bn)(w_h-\mu_h),(w_h-\mu_h)\r_{\partial\cT_h},\nonumber\\
    &a_{2,h}((\br_h, w_h, \mu_h),(\br_h, w_h, \mu_h))\\
    &=(\br_h,\br_h)_{\cT_h}+((\gamma-\frac12\nabla\cdot\bm{\zeta})w_h, w_h)_{\cT_h}+\l(\tau_2+\frac12\bz\cdot\bn)(w_h-\mu_h),(w_h-\mu_h)\r_{\partial\cT_h}.\nonumber
  \end{alignat*}
\end{lemma}

\begin{proof}
  We prove the identity involves $a_{1,h}$, the one involves $a_{2,h}$ is similar. It follows from \eqref{eq:hdgah1} and integration by parts that
  \begin{equation*}
\begin{aligned}
    &a_{1,h}((\br_h, w_h, \mu_h),(\br_h, w_h, \mu_h))\\
    =&(\br_h,\br_h)_{\cT_h}-(w_h,\nabla\cdot\br_h)_{\cT_h}+\l \mu_h,\br_h\cdot\bn\r_{\partial\cT_h}\\
    &+(\nabla\cdot\br_h ,w_h)_{\cT_h}-(w_h,\bm{\zeta}\cdot \nabla w_h)_{\cT_h}+((\gamma-\nabla\cdot\bm{\zeta})w_h, w_h)_{\cT_h}\\
  &+\l\mu_h\bz\cdot\bn +\tau_1(w_h-\mu_h),w_h\r_{\partial\cT_h}-\l(\br_h+\mu_h\bz)\cdot\bn +\tau_1(w_h-\mu_h),\mu_h\r_{\partial\cT_h}\\
  =&(\br_h,\br_h)_{\cT_h}+((\gamma-\frac12\nabla\cdot\bm{\zeta})w_h, w_h)_{\cT_h}\\
  &+\l(\tau_1-\frac12\bz\cdot\bn)(w_h-\mu_h),(w_h-\mu_h)\r_{\partial\cT_h}-\frac12\l\bz\cdot\bn\mu_h,\mu_h\r_{\partial\cT_h}\\
  =&(\br_h,\br_h)_{\cT_h}+((\gamma-\frac12\nabla\cdot\bm{\zeta})w_h, w_h)_{\cT_h}+\l(\tau_1-\frac12\bz\cdot\bn)(w_h-\mu_h),(w_h-\mu_h)\r_{\partial\cT_h},
\end{aligned}
\end{equation*}
where we use $\l\bz\cdot\bn\mu_h,\mu_h\r_{\partial\cT_h}=0$.
\end{proof}

According to Lemma \ref{lem:ahnorm} and assumptions \eqref{assump2} and \eqref{assump3}, we define a scaled energy norm as follows,
\begin{equation}\label{eq:enorm}
  \|(\br,w,\mu)\|_{1,\beta}^2=\beta^\frac12(\|\br\|^2_{\cT_h}+\|w\|^2_{\cT_h}+\||\tau_1-\frac12\bm{\zeta}\cdot\bn|^\frac12(w-\mu)\|_{\partial\cT_h}^2)+\|w\|^2_{\cT_h}.
\end{equation}

\subsection{An inf-sup condition}
Using \eqref{eq:hdgahduala}, we can replace the bilinear form $\cB_h$ in the HDG discretization \eqref{eq:hdgconcise} as 
\begin{equation}
\begin{aligned}
  &\cB_h(((\bm{q}_h, y_h, \widehat{y}_h),(\bm{p}_h, p_h, \widehat{p}_h)),((\bm{r}_1,w_1,\mu_1),(\bm{r}_2,w_2,\mu_2)))\\
  =&\beta^\frac12 a_{1,h}((-\bm{r}_2,w_2,\mu_2),(-\bm{p}_h, p_h, \widehat{p}_h))+( y_h,w_2)_{\cT_h}\\
  &-( p_h,w_1)_{\cT_h}+\beta^\frac12a_{1,h}((\bm{q}_h, y_h, \widehat{y}_h),(\bm{r}_1,w_1,\mu_1)).
\end{aligned}
\end{equation}

\begin{lemma}[Inf-sup] \label{lem:infsup}
Under Assumption \ref{assump:tau}, for any $((\bm{q}_h, y_h, \widehat{y}_h),(\bm{p}_h, p_h, \widehat{p}_h))$ in $(\bm{V}_h\times W_h\times\Lambda^0_h)\times(\bm{V}_h\times W_h\times\Lambda^0_h)$, we have
\begin{equation*}
\begin{aligned}
   &\|(\bm{p}_h, p_h, \widehat{p}_h)\|_{1,\beta}+\|(\bm{q}_h, y_h, \widehat{y}_h)\|_{1,\beta}\\
   \lesssim&\sup_{\substack{(\br_i,w_i,\mu_i)\in\bm{V}_h\times W_h\times\Lambda^0_h\\ i=1,2}}\frac{\cB_h(((\bm{q}_h, y_h, \widehat{y}_h),(\bm{p}_h, p_h, \widehat{p}_h)),((\bm{r}_1,w_1,\mu_1),(\bm{r}_2,w_2,\mu_2)))}{\|(\br_1,w_1,\mu_1)\|_{1,\beta}+\|(\br_2,w_2,\mu_2)\|_{1,\beta}}.
\end{aligned}
\end{equation*}
\end{lemma}
\begin{proof}
    Given $((\bm{q}_h, y_h, \widehat{y}_h),(\bm{p}_h, p_h, \widehat{p}_h))\in (\bm{V}_h\times W_h\times\Lambda^0_h)\times(\bm{V}_h\times W_h\times\Lambda^0_h)$, take $(\bm{r}_1,w_1,\mu_1)=(\bm{q}_h-\bm{p}_h, y_h- p_h, \widehat{y}_h-\widehat{p}_h)$ and $(\bm{r}_2,w_2,\mu_2)=(\bm{q}_h+\bm{p}_h, y_h+ p_h, \widehat{y}_h+\widehat{p}_h)$. A simple calculation shows that
  \begin{equation*}
    \begin{aligned}
      &\cB_h(((\bm{q}_h, y_h, \widehat{y}_h),(\bm{p}_h, p_h, \widehat{p}_h)),((\bm{r}_1,w_1,\mu_1),(\bm{r}_2,w_2,\mu_2)))\\
  =&\beta^\frac12 a_{1,h}((\bm{p}_h, p_h, \widehat{p}_h),(\bm{p}_h, p_h, \widehat{p}_h))+( y_h, y_h)_{\cT_h}+( p_h, p_h)_{\cT_h}\\
  &+\beta^\frac12a_{1,h}((\bm{q}_h, y_h, \widehat{y}_h),(\bm{q}_h, y_h, \widehat{y}_h))\\
  \gtrsim&\|(\bm{p}_h, p_h, \widehat{p}_h)\|_{1,\beta}^2+\|(\bm{q}_h, y_h, \widehat{y}_h)\|_{1,\beta}^2,
    \end{aligned}
  \end{equation*}
  where we use the assumption \eqref{eq:advassump}.
It follows from the parallelogram law \cite{BS} that
\begin{equation*}
\begin{aligned}
   &\|(\bm{q}_h-\bm{p}_h, y_h- p_h, \widehat{y}_h-\widehat{p}_h)\|_{1,\beta}^2+\|(\bm{q}_h+\bm{p}_h, y_h+ p_h, \widehat{y}_h+\widehat{p}_h)\|^2_{1,\beta}\\
  =&2(\|(\bm{p}_h, p_h, \widehat{p}_h)\|_{1,\beta}^2+\|(\bm{q}_h, y_h, \widehat{y}_h)\|_{1,\beta}^2).
\end{aligned}
\end{equation*}
This finishes the proof.
\end{proof}

\begin{remark}
  The inf-sup condition guarantees the well-posedness of the HDG method \eqref{eq:hdgconcise}, or equivalently \eqref{eq:hdg2}, by the standard saddle point theory in \cite{Brezzi,Bab}. 
\end{remark}

\subsection{Concrete error estimates}

It is well-known that HDG methods are consistent (cf. \cite{fu2015analysis}), hence we have the following Galerkin orthogonality. Let $({\bq},{\bp},{y},{p})$ be the solution of \eqref{eq:fsystem} and let $(\bq_h,\bp_h,y_h,p_h,\widehat{y},\widehat{p}_h)$ be the HDG solution of \eqref{eq:hdg2}. We have, 
\begin{equation}\label{eq:go}
   \cB_h(((\bq-\bm{q}_h,y- y_h, y-\widehat{y}_h),(\bm{p}-\bm{p}_h,p- p_h, p-\widehat{p}_h)),((\br_1,w_1,\mu_1),(\br_2,w_2,\mu_2)))=0,
\end{equation}
for all $((\bm{r}_1,w_1,\mu_1),(\bm{r}_2,w_2,\mu_2))\in (\bm{V}_h\times W_h\times\Lambda^0_h)\times(\bm{V}_h\times W_h\times\Lambda^0_h)$.
\begin{lemma}
  We have,
  \begin{equation}\label{eq:epsih}
  \begin{aligned}
     \|(\epsilon_h^{\bq}&,\epsilon_h^{y},\epsilon_h^{\widehat{y}})\|_{1,\beta}+\|(\epsilon_h^{\bm{p}},\epsilon_h^{p},\epsilon_h^{\widehat{p}})\|_{1,\beta}\\
  &\le C(\beta^\frac14h^{k+\frac12}+h^{k+1})(|p|_{k+1}+|\bp|_{k+1}+|y|_{k+1}+|\bq|_{k+1}),
  \end{aligned}
  \end{equation}
where we use the notations defined in \eqref{eq:projnotations}.
\end{lemma}

\begin{proof}
  Since $(\epsilon_h^{\bq},\epsilon_h^{y},\epsilon_h^{\widehat{y}})\in\bm{V}_h\times W_h\times\Lambda^0_h$ and $(\epsilon_h^{\bm{p}},\epsilon_h^{p},\epsilon_h^{\widehat{p}})\in\bm{V}_h\times W_h\times\Lambda^0_h$, it follows from Lemma \ref{lem:infsup} and \eqref{eq:go} that
  \begin{equation}\label{eq:errorinfsup}
  \begin{aligned}
   &\|(\epsilon_h^{\bq},\epsilon_h^{y},\epsilon_h^{\widehat{y}})\|_{1,\beta}+\|(\epsilon_h^{\bm{p}},\epsilon_h^{p},\epsilon_h^{\widehat{p}})\|_{1,\beta}\\
   \lesssim&\sup_{\substack{(\br_i,w_i,\mu_i)\in\bm{V}_h\times W_h\times\Lambda^0_h\\ i=1,2}}\frac{\cB_h(((\epsilon_h^{\bq},\epsilon_h^{y},\epsilon_h^{\widehat{y}}),(\epsilon_h^{\bm{p}},\epsilon_h^{p},\epsilon_h^{\widehat{p}})),((\bm{r}_1,w_1,\mu_1),(\bm{r}_2,w_2,\mu_2)))}{\|(\br_1,w_1,\mu_1)\|_{1,\beta}+\|(\br_2,w_2,\mu_2)\|_{1,\beta}}\\
   =&\sup_{\substack{(\br_i,w_i,\mu_i)\in\bm{V}_h\times W_h\times\Lambda^0_h\\ i=1,2}}\frac{\cB_h(((\delta^{\bq}_h,\delta^{y}_h,\delta^{\widehat{y}}_h),(\delta^{\bp}_h,\delta^{p}_h,\delta^{\widehat{p}}_h)),((\bm{r}_1,w_1,\mu_1),(\bm{r}_2,w_2,\mu_2)))}{\|(\br_1,w_1,\mu_1)\|_{1,\beta}+\|(\br_2,w_2,\mu_2)\|_{1,\beta}}.
  \end{aligned}
  \end{equation}
  Notice that we have the following relation by the definition of the projection operators \eqref{eq:hdgproj}, \eqref{eq:hdgprojor} and integration by parts (also see \cite[Lemma 4.5]{fu2015analysis}),
  \begin{equation*}
\begin{aligned}
&a_{1,h}((\delta^{\bq}_h,\delta^{y}_h,\delta^{\widehat{y}}_h),(\bm{r}_1,w_1,\mu_1))\\
=&(\delta^{\bq}_h,\br_1)_{\cT_h}-(\delta^{y}_h,\nabla\cdot\br_1)_{\cT_h}+\l \delta^{\widehat{y}}_h,\br_1\cdot\bn\r_{\partial\cT_h}\\
    &+(\nabla\cdot\delta^{\bq}_h ,w_1)_{\cT_h}-(\delta^{y}_h,\bm{\zeta}\cdot \nabla w_1)_{\cT_h}+((\gamma-\nabla\cdot\bm{\zeta})\delta_h^y,w_1)_{\cT_h}\\
  &+\l\delta^{\widehat{y}}_h\bz\cdot\bn +\tau_1(\delta^{y}_h-\delta^{\widehat{y}}_h),w_1\r_{\partial\cT_h}-\l(\delta^{\bq}_h+\delta^{\widehat{y}}_h\bz)\cdot\bn +\tau_1(\delta^{y}_h-\delta^{\widehat{y}}_h),\mu_1\r_{\partial\cT_h}\\
  =&-(\bm{\zeta}\delta^{y}_h,\nabla w_1)_{\cT_h}+((\gamma-\nabla\cdot\bm{\zeta})\delta_h^y,w_1)_{\cT_h}\\
  &+\l\delta^{\bq}_h\cdot\bn,w_1-\mu_1\r_{\partial\cT_h}+\l\delta^{\widehat{y}}_h\bz\cdot\bn ,w_1\r_{\partial\cT_h}+\l\tau_1\delta^{y}_h,w_1-\mu_1\r_{\partial\cT_h}\\
  :=&R_1+R_2+\ldots+R_5,
\end{aligned}
\end{equation*}
where we use Assumption \eqref{assump1} and the fact $\l\delta^{\widehat{y}}_h\bz\cdot\bn ,\mu_1\r_{\partial\cT_h}=0$. 
The terms $R_1$ to $R_5$ can be bounded as follows by the projection estimates \eqref{eq:projection}. 
 Let $\l\bm{\zeta}\r_K$ be the mean of $\bm{\zeta}$ over each $K$. Note that $\l\bm{\zeta}\r_K\cdot\nabla w_1\in\mathbb{P}^{k-1}(K)$ hence $(\delta_h^y\l\bm{\zeta}\r_K, \nabla w_1)_{\cT_h}=0$. We also have $\|\bm{\zeta}-\l\bm{\zeta}\r_K\|_{L^\infty(K)}\le C h_K$. Then it follows from a standard inverse inequality that
\begin{equation}\label{eq:R1}
\begin{aligned}
   R_1&=\big((\l\bm{\zeta}\r_K-\bm{\zeta})\delta_h^y,\nabla w_1\big)_{\cT_h}\le Ch\|\delta_h^y\|_{\cT_h}\|\nabla w_1\|_{\cT_h}\\
   &\le C\|\delta_h^y\|_{\cT_h}\|w_1\|_{\cT_h}\le Ch^{k+1}|y|_{k+1}\|w_1\|_{\cT_h}.
\end{aligned}
\end{equation}
For $R_2$, we have
\begin{equation}
  R_2\le (\|\gamma\|_\infty+|\bm{\zeta}|_{1,\infty})\|\delta_h^y\|_{\cT_h}\|w_1\|_{\cT_h}\le Ch^{k+1}|y|_{k+1}\|w_1\|_{\cT_h}.
\end{equation}
It follows from assumption \eqref{assump3} and trace inequalities that 
\begin{equation}
  \begin{aligned}
    R_3&\le \left\|\big|\tau_1-\frac12\bm{\zeta}\cdot\bn\big|^{-\frac{1}{2}}\delta_h^{\bq}\right\|_{\partial\cT_h}\left\|\big|\tau_1-\frac12\bm{\zeta}\cdot\bn\big|^\frac12(w_1-\mu_1)\right\|_{\partial\cT_h}\\
    &\le Ch^{k+\frac12}|\bm{q}|_{k+1}\left\|\big|\tau_1-\frac12\bm{\zeta}\cdot\bn\big|^\frac12(w_1-\mu_1)\right\|_{\partial\cT_h}.
  \end{aligned}
\end{equation}
The term $R_4$ can be estimated similarly as $R_1$. We have, by trace inequalities, 
\begin{equation}
\begin{aligned}
  R_4&=\l(\l\bm{\zeta}\r_K-\bm{\zeta})\cdot\bn\delta_h^{\widehat{y}},w_1\r_{\partial\cT_h}\\
  &\le Ch\|\delta_h^{\widehat{y}}\|_{\partial\cT_h}\|w_1\|_{\partial\cT_h}\\
  &\le Ch^\frac12\|\delta_h^{\widehat{y}}\|_{\partial\cT_h}\|w_1\|_{\cT_h}\\
  &\le C h^{k+1}|y|_{k+1}\|w_1\|_{\cT_h}.
\end{aligned}
\end{equation}
At last, it follows from assumptions \eqref{assump1} and \eqref{assump3} that
\begin{equation}\label{eq:R5}
  \begin{aligned}
    R_5&\le\|\tau_1^\frac12\delta_h^y\|_{\partial\cT_h}\|\tau_1^\frac12(w_1-\mu_1)\|_{\partial\cT_h}\\
    &\le C\|\delta_h^y\|_{\partial\cT_h}\left\|\big|\tau_1-\frac12\bm{\zeta}\cdot\bn\big|^\frac12(w_1-\mu_1)\right\|_{\partial\cT_h}\\
    &\le Ch^{k+\frac12}|y|_{k+1}\left\|\big|\tau_1-\frac12\bm{\zeta}\cdot\bn\big|^\frac12(w_1-\mu_1)\right\|_{\partial\cT_h}.
  \end{aligned}
\end{equation}
Combining the estimates \eqref{eq:R1}-\eqref{eq:R5}, we obtain,
\begin{equation}\label{eq:a1hproj}
  {\small
\begin{aligned}
&a_{1,h}((\delta^{\bq}_h,\delta^{y}_h,\delta^{\widehat{y}}_h),(\bm{r}_1,w_1,\mu_1))\\
  &\le Ch^{k+\frac12}(|y|_{k+1}+|\bq|_{k+1})\left(\|w_1\|_{\cT_h}+\left\|\big|\tau_1-\frac12\bm{\zeta}\cdot\bn\big|^\frac12(w_1-\mu_1)\right\|_{\partial\cT_h}\right).
\end{aligned}}
\end{equation}
Similarly, we have,
\begin{equation}\label{eq:a2hproj}
{\small
\begin{aligned}
  &a_{2,h}((\delta^{\bp}_h,\delta^{p}_h,\delta^{\widehat{p}}_h),(\bm{r}_2,w_2,\mu_2))\\
  &\le Ch^{k+\frac12}(|p|_{k+1}+|\bp|_{k+1})\left(\|w_2\|_{\cT_h}+\left\|\big|\tau_2+\frac12\bm{\zeta}\cdot\bn\big|^\frac12(w_2-\mu_2)\right\|_{\partial\cT_h}\right).
\end{aligned}}
\end{equation}
Note that $\tau_2$ may not be piecewise constant here.
  Hence, it follows from \eqref{eq:a1hproj} and \eqref{eq:a2hproj} that
\begin{equation}\label{eq:bheq}
\begin{aligned}
    &\cB_h(((\delta^{\bq}_h,\delta^{y}_h,\delta^{\widehat{y}}_h),(\delta^{\bp}_h,\delta^{p}_h,\delta^{\widehat{p}}_h)),((\bm{r}_1,w_1,\mu_1),(\bm{r}_2,w_2,\mu_2)))\\
    =&\beta^\frac12 a_{2,h}((\delta^{\bp}_h,\delta^{p}_h,\delta^{\widehat{p}}_h),(\bm{r}_2,w_2,\mu_2))+(\delta^{y}_h,w_2)_{\cT_h}\\
  &-(\delta^p_h,w_1)_{\cT_h}+\beta^\frac12a_{1,h}((\delta^{\bq}_h,\delta^{y}_h,\delta^{\widehat{y}}_h),(\bm{r}_1,w_1,\mu_1))\\
  \le&C(\beta^\frac14h^{k+\frac12}+h^{k+1})(|p|_{k+1}+|\bp|_{k+1}+|y|_{k+1}+|\bq|_{k+1})
  \\&\times(\|(\br_1,w_1,\mu_1)\|_{1,\beta}+\|(\br_2,w_2,\mu_2)\|_{1,\beta}).
\end{aligned}
\end{equation}
Therefore, we obtain the desired result by \eqref{eq:errorinfsup}, \eqref{eq:bheq} and \eqref{eq:projection}.
\end{proof}

\begin{theorem}\label{thm:hdgesti}
  Let $(({\bm{q}}_h,y_h, \widehat{y}_h),({\bp}_h, p_h, \widehat{p}_h))$ be the HDG solutions to \eqref{eq:hdg2} and let $(\bq,y)$ and $(\bp,p)$ be the solutions to \eqref{eq:fsystem}, we have the following error estimates
  \begin{equation*}
  \begin{aligned}
    &\|(\bq-\bq_h,y-y_h,y-\widehat{y}_h)\|_{1,\beta}+\|(\bp-\bp_h,p-p_h,p-\widehat{p}_h)\|_{1,\beta}\\
    \le& C\big(\beta^{\frac14}h^{k+\frac12}+h^{k+1}\big)(|p|_{k+1}+|\bp|_{k+1}+|y|_{k+1}+|\bq|_{k+1}).
  \end{aligned}
  \end{equation*}
\end{theorem}
\begin{proof}
  First, it follows from \eqref{eq:enorm} and the approximation properties of the projection operators that,
  \begin{equation}\label{eq:deltah}
  \begin{aligned}
        &\|(\delta^{\bq}_h,\delta^{y}_h,\delta^{\widehat{y}}_h)\|_{1,\beta}+\|(\delta^{\bp}_h,\delta^{p}_h,\delta^{\widehat{p}}_h)\|_{1,\beta}\\
    \le& C\big(\beta^{\frac14}h^{k+\frac12}+h^{k+1}\big)(|p|_{k+1}+|\bp|_{k+1}+|y|_{k+1}+|\bq|_{k+1}).
  \end{aligned}
  \end{equation}
By \eqref{eq:deltah}, \eqref{eq:epsih} and triangle inequality, we obtain
\begin{equation*}
\begin{aligned}
   &\|(\bq-\bq_h,y-y_h,y-\widehat{y}_h)\|_{1,\beta}+\|(\bp-\bp_h,p-p_h,p-\widehat{p}_h)\|_{1,\beta}\\
    \le& C\big(\beta^{\frac14}h^{k+\frac12}+h^{k+1}\big)(|p|_{k+1}+|\bp|_{k+1}+|y|_{k+1}+|\bq|_{k+1}).
    \end{aligned}
  \end{equation*}
\end{proof}

\begin{remark}
  For sufficiently smooth solutions, Theorem \ref{thm:hdgesti} states that the convergence rate in the $\|\cdot\|_{1,\beta}$ norm is $O(h^{k+\frac12})$ when $\beta=O(1)$, while it is $O(h^{k+1})$ when $\beta=O(h^2)$. However, if the right-hand sides of \eqref{eq:gsaddle} are only in $L_2(\Omega)$, as discussed in Remark \ref{rem:regularity}, the regularity of the solutions depends on $\beta$. In this case, one would expect to see worsened convergence rates in $\|\cdot\|_{1,\beta}$ for small $\beta$ with a coarse mesh. Therefore, a very fine mesh would be required to obtain meaningful solutions (cf. \cite{brenner2020multigrid}). Consequently, a robust preconditioner or a fast solver is necessary to expedite the solving process.
\end{remark}

\section{A BDDC Preconditioner}\label{sec:bddc}

In this section, we introduce a BDDC preconditioner to solve the discrete problem \eqref{eq:hdg2}. To achieve this, we first write the discrete problem in an operator form and provide several preliminary results. For simplicity, we only consider the two-dimensional case.

\subsection{Operator form}

We define the following operators, for any $\bm{r}_1\in \bm{V}_h$, $w_1, w_2 \in W_h $, and $\mu_1, \mu_2\in \Lambda^0_h$,
\begin{equation*}
  \begin{aligned}
&(\cA {\qb}_h,\bm{r}_1)_{\cT_h}=-\beta^{\frac12}({\qb}_h,\bm{r}_1)_{\cT_h},\quad(\cB\bm{r}_1,\ty_h)_{\cT_h}=\beta^{\frac12}(\ty_h,\nabla\cdot\bm{r}_1)_{\cT_h},\\
&\l\cC\bm{r}_1 ,\widehat{y}_h\r_{\cT_h}=-\beta^{\frac12}\l\widehat{y}_h,\bm{r}\cdot\bm{n}\r_{\partial\cT_h},\\
&(\mathcal{R}_1 {\ty}_h,{w}_1)_{\cT_h} =-\beta^{\frac12}(\ty_h,\bz\cdot\nabla w_1)_{\cT_h}+\beta^{\frac12}((\gamma-\nabla\cdot\bm{\zeta})\ty_h,w_1)_{\cT_h}+\beta^{\frac12}\l\tau_1\ty_h,w_1\r_{\partial\cT_h}, \\
&(\mathcal{R}_4p_h,w_2)_{\cT_h}=\beta^{\frac12}(p_h,\bz\cdot\nabla w_2)_{\cT_h}+\beta^{\frac12}(\gamma p_h,w_2)_{\cT_h}+\beta^{\frac12}\l\tau_2p_h,w_2\r_{\partial\cT_h},\\
&(\mathcal{R}_2p_h,w_1)_{\cT_h}=-(p_h,w_1)_{\cT_h},\quad(\mathcal{R}_3y_h,w_2)_{\cT_h}=(y_h,w_2)_{\cT_h},\\
&(\cS_1\widehat{y}_h,w_1)_{\cT_h}=\beta^{\frac12}\l(\bz\cdot\bm{n}-\tau_1)\widehat{y}_h,w_1\r_{\partial\cT_h},\\
&\l\cS_2\widehat{p}_h,w_2\r_{\partial\cT_h}=-\beta^{\frac12}\l(\bz\cdot\bm{n}+\tau_2)\widehat{p}_h,w_2\r_{\partial\cT_h},\\
&\l\cS_3y_h,\mu_1\r_{\partial\cT_h}=-\beta^{\frac12}\l\tau_1y_h,\mu_1\r_{\partial\cT_h}, \quad \l\cS_4p_h,\mu_2\r_{\partial\cT_h}=-\beta^{\frac12}\l\tau_2p_h,\mu_2\r_{\partial\cT_h},\\
&\l\cT_1\widehat{y}_h,\mu_1\r_{\partial\cT_h}=-\beta^{\frac12}\l(\bz\cdot\bm{n}-\tau_1)\widehat{y}_h,w_1\r_{\partial\cT_h},\\
&\l\cT_2\widehat{y}_h,\mu_2\r_{\partial\cT_h}=\beta^{\frac12}
\l(\bz\cdot\bm{n}+\tau_2)\widehat{p}_h,\mu_2)\r_{\partial\cT_h}.
\end{aligned}
\end{equation*}
Consequently, the HDG method \eqref{eq:hdg2} can be written as the following operator form, which is useful for the description of the BDDC algorithm,
\begin{equation}\label{equation:Amatrix}
\begin{aligned}
\begin{pmatrix}
\mathcal{A} & 0 & \mathcal{B}^T & 0&\mathcal{C}^T &0\\
0& \mathcal{A}  & 0 & \mathcal{B}^T& 0&\mathcal{C}^T\\
\cB & 0 & \mathcal{R}_1 & \mathcal{R}_2 & \cS_1 &0\\
0 & \cB &\mathcal{R}_3 &\mathcal{R}_4 &0 & S_2\\
\cC&0 & \cS_3 &  0 & \cT_1& 0\\
0&\cC &  0 &S_4   &0 & \cT_2
\end{pmatrix}
\begin{pmatrix}
\bm{q}_h\\
\bm{p}_h\\
y_h\\
p_h\\
\widehat{y}_h\\
\widehat{p}_h
\end{pmatrix}
=\begin{pmatrix}
0\\
0\\
g\\
f\\
0\\
0
\end{pmatrix}.
\end{aligned}
\end{equation}

\subsection{A characterization of the HDG Method}

The operator form $\eqref{equation:Amatrix}$ can be rewritten as the following matrix-vector form:
\begin{align}\label{Aqu}
\begin{pmatrix}
{A}_{\bm{q}\bm{q}}&\bm{0}&{A}^T_{y\bm{q}}&\bm{0}&{A}^T_{\widehat{y}\bm{q}} &\bm{0}\\
\bm{0}&{A}_{\bm{p}\bm{p}}&\bm{0}& {A}^T_{p\bm{p}}&
\bm{0} & {A}^T_{\widehat{p}\bm{p}}\\
{A}_{y\bm{q}}&\bm{0}&{A}_{yy} &{A}_{yp}&{A}_{y\widehat{y}}&\bm{0}\\
\bm{0}&{A}_{p\bm{p}}&{A}_{py}&{A}_{pp} &\bm{0}&{A}_{p\widehat{p}}\\
{A}_{\widehat{y}\bm{q}} &\bm{0}& {A}_{\widehat{y}{\ty}}&\bm{0}& A_{\widehat{y}\widehat{y}} &\bm{0}\\
\bm{0}& {A}_{\widehat{p}\bm{p}}&\bm{0}&A_{\widehat{p}p}&\bm{0}&A_{\widehat{p}\widehat{p}} 
\end{pmatrix}
\begin{pmatrix}
\bm{q}_c\\
\bm{p}_c\\
y_c\\
p_c\\
\widehat{y}_c\\
\widehat{p}_c
\end{pmatrix}
=\begin{pmatrix}
\bm{0}\\
\bm{0}\\
\bm{G}\\
\bm{F}\\
\bm{0}\\
\bm{0}
\end{pmatrix},
\end{align}
where $\bm{F}$ and $\bm{G}$ are vectors of the coefficients that represent $f$ and $g$ under the corresponding polynomial basis. Similarly, the vector $(\bm{q}_c, \bm{p}_c, y_c, p_c, \widehat{y}_c, \widehat{p}_c)$ contains the coefficients of $(\bm{q}_h,\bm{p}_h,y_h,p_h,\widehat{y}_h,\widehat{p}_h)$ under the corresponding polynomial basis. The matrices $A$ in \eqref{Aqu} with double indices represent the corresponding operators in \eqref{equation:Amatrix}, where each index specifies the corresponding basis being used.
\vspace{0.3cm}

Furthermore, we define $\bm{L}=\begin{pmatrix}
\bm{q}_c\\
\bm{p}_c
\end{pmatrix}$, $\bm{u}=\begin{pmatrix}
y_c\\
p_c
\end{pmatrix}$, $\bm{\lambda}=\begin{pmatrix}
\widehat{y}_c\\
\widehat{p}_c
\end{pmatrix}$ and $\bm{\bm{Z}}=\begin{pmatrix}
\bm{G}\\
\bm{F}
\end{pmatrix}$. We then decompose the system matrix in \eqref{Aqu} as follow. Let
\begin{equation*}
\begin{alignedat}{3}
  &\A_{\bm{L}\bm{L}}=\begin{pmatrix}
{A}_{\bm{q}\bm{q}}&\bm{0}\\
\bm{0}&{A}_{\bm{p}\bm{p}}
\end{pmatrix}\quad&&\A_{\bm{u}\bm{L}}=\begin{pmatrix}
{A}_{y\bm{q}}&\bm{0}\\
\bm{0}&{A}_{p\bm{p}}\\
\end{pmatrix} \quad&&\A_{\lambdae\bm{L}}=\begin{pmatrix}
{A}_{\widehat{y}{\qb}} &\bm{0}\\
\bm{0}& {A}_{\widehat{p}\bm{p}}
\end{pmatrix}\\
&\A_{{\u}{\u}}=\begin{pmatrix}
{A}_{yy} &{A}_{yp}\\
{A}_{py}&{A}_{pp} 
\end{pmatrix}\quad &&\A_{\bm{u}\bm{\lambda}}=\begin{pmatrix}
{A}_{y\widehat{y}}&0\\
0&{A}_{p\widehat{p}}\\
\end{pmatrix} \quad&&\A_{\lambdae\bm{u}}=\begin{pmatrix}
 A_{\widehat{y}y} & 0\\
0&A_{\widehat{p}p}
\end{pmatrix}\\
&\A_{\lambdae\lambdae}=\begin{pmatrix}
A_{\widehat{y}\widehat{y}} &0\\
0&A_{\widehat{p}\widehat{p}} 
\end{pmatrix},
\end{alignedat}
\end{equation*}
then \eqref{Aqu} can be written as 
\begin{align}\label{ALu}
\begin{pmatrix}
\A_{\L\L} &\A^T_{\u\L} & \A_{\lambdae\L}^T\\
\A_{\u\L} &\A_{\u\u} &\A_{\u\lambdae}\\
\A_{\lambdae\L}&\A_{\lambdae\u}&\A_{\lambdae\lambdae}
\end{pmatrix}
\begin{pmatrix}
\L\\
\u\\
\lambdae
\end{pmatrix}=\begin{pmatrix}
\bm{0}\\
\bm{Z}\\
\bm{0}
\end{pmatrix}.
\end{align}
We can eliminate ${\L}$ and $\u$ in each element independently from \eqref{ALu} and obtain a system for 
$\lambdae$ as
\begin{align}\label{Ab}
\A\lambdae={\bm b},
\end{align}
where 
\begin{align*}
\A=\A_{\lambdae\lambdae}-
\begin{pmatrix}
\A_{\lambdae\L}&\A_{\lambdae\u}
\end{pmatrix}\begin{pmatrix}
\A_{\L\L}&\A^{T}_{\u\L}\\
\A_{\u\L}&\A_{\u\u}
\end{pmatrix}^{-1}
\begin{pmatrix}
\A^T_{\lambdae\L}\\
\A_{\u\lambdae}
\end{pmatrix}
\end{align*}
and 
\begin{align*}
\bm{b}=-
\begin{pmatrix}
\A_{\lambdae\L}&\A_{\lambdae\u}
\end{pmatrix}
\begin{pmatrix}
\A_{\L\L} &\A^{T}_{\u\L}\\
\A_{\u\L}&\A_{\u\u}
\end{pmatrix}^{-1}
\begin{pmatrix}
\bm{0}\\
\bm{Z}
\end{pmatrix}.
\end{align*}
We will solve for $\lambdae$ based on \eqref{Ab}. Then ${\L},{\u}$ can be recovered in each element $K$ with $\lambdae$ on $\partial K$ using \eqref{ALu}.

\begin{remark}
  One of the key features of HDG is the use of a technique known as static condensation, which reduces the problem \eqref{ALu} into \eqref{Ab}. This approach allows us to solve the discrete system on the skeleton, significantly reducing the degrees of freedom. For further details, we refer the reader to \cite{fu2015analysis,MR2429868,chen2018hdg}.
\end{remark}

\subsection{Domain decomposition and a reduced subdomain interface problem}
In this subsection, we briefly discuss the construction of a reduced subdomain interface problem that we aim to solve. Similar procedure can be found in \cite{TZAD2021,TWZstokes2020}. 

We first decompose the domain $\Omega$ into $N$ non-overlapping subdomains $\Omega_i (i=1,2,\cdots,N)$ and denote the diameter of each subdomain as $H_i$. Let $H=\max\limits_{i}H_i$ and $\Gamma=\cup\partial\Omega^{(i)}\backslash\partial\Omega$ be the subdomain interface.
Denote $\widehat{\Lambdae}_\Gamma:=\{\lambdae_\Gamma\}$ as the set of degrees of freedom on the subdomain interface $\Gamma$ and $\Lambdae^{(i)}_I$ as the set of degrees of freedom in the interior of each subdomain. We denote $\Lambdae_I=\bigoplus_{i=1}^N \Lambdae^{(i)}_I:=\{\lambdae_I\}$. See Figure \ref{figure:lambdagi} for an illustration. Then, we have that
\begin{align*}
\Lambdae=\Lambdae_I\bigoplus \widehat{\Lambdae}_\Gamma.
\end{align*}

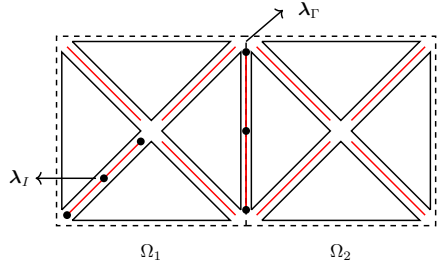
\begin{figure}[!h]
\centering
\scalebox{0.7}{
\begin{tikzpicture}

 \draw[thick] plot coordinates {(0,-0.2) (1.5,1.3) (3,-0.2) (0,-0.2)};
  \draw[thick] plot coordinates {(3.2,0) (3.2,3) (1.7,1.5)(3.2,0)};
 \draw[thick] plot coordinates {(0,3.2)(3,3.2) (1.5,1.7)(0,3.2)};
 \draw[thick] plot coordinates {(-0.2,0)(-0.2,3) (1.3,1.5)(-0.2,0)};

\draw[thick] plot coordinates {(3.6,-0.2) (5.1,1.3) (6.6,-0.2) (3.6,-0.2)};
  \draw[thick] plot coordinates {(6.8,0) (6.8,3) (5.3,1.5)(6.8,0)};
 \draw[thick] plot coordinates {(3.6,3.2)(6.6,3.2) (5.1,1.7)(3.6,3.2)};
 \draw[thick] plot coordinates {(3.4,0)(3.4,3) (4.9,1.5)(3.4,0)};

 \draw[thick,red] plot coordinates {(1.7,1.7) (3.1,3.1)};
 \draw[thick,red] plot coordinates {(1.3,1.3) (-0.1,-0.1)};
 \draw[thick,red] plot coordinates {(1.7,1.3) (3.1,-0.1)};
 \draw[thick,red] plot coordinates {(1.3,1.7) (-0.1,3.1)};

 \draw[thick,red] plot coordinates {(5.3,1.7) (6.7,3.1)};
 \draw[thick,red] plot coordinates {(4.9,1.3) (3.5,-0.1)};
 \draw[thick,red] plot coordinates {(5.3,1.3) (6.7,-0.1)};
 \draw[thick,red] plot coordinates {(4.9,1.7) (3.5,3.1)};

 \draw[thick,red] plot coordinates {(3.3,0) (3.3,3)};


  \draw[thick,dashed] plot coordinates {  (3.3, 3.3) (-0.3,3.3)(-0.3,-0.3) (3.3, -0.3)};

  \draw[thick,dashed] plot coordinates {  (3.3, -0.3)  (3.3, 3.3)};

  \draw[thick,dashed] plot coordinates {(3.3, -0.3) (6.9, -0.3) (6.9, 3.3) (3.3, 3.3)};

 \draw[thick,red] plot coordinates {(3.3,0) (3.3,3)};

\node at (1.5,-0.8) {$\Omega_1$};
\node at (5.1,-0.8) {$\Omega_2$};

\node at (3.3,0) [circle,fill,inner sep=1.5pt]{};
\node at (3.3,3) [circle,fill,inner sep=1.5pt]{};
\node at (3.3,1.5) [circle,fill,inner sep=1.5pt]{};

\node at (1.3,1.3) [circle,fill,inner sep=1.5pt]{};
\node at (-0.1,-0.1) [circle,fill,inner sep=1.5pt]{};
\node at (0.6,0.6) [circle,fill,inner sep=1.5pt]{};

\draw[->,thick] (3.3,3.2)--(4,3.8);
\node at (4.5,3.8) {$\bm{\lambda}_\Gamma$};

\draw[->,thick] (0.45,0.6)--(-0.7,0.6);
\node at (-1.0,0.6) {$\bm{\lambda}_I$};

\end{tikzpicture}}
\caption{The degree of freedoms $\lambdae_{\Gamma}$ and $\lambdae_{I}$} \label{figure:lambdagi}
\end{figure}

The original global problem \eqref{Ab} can then be written as
\begin{align*}
\begin{pmatrix}
{\A}_{II}& {\A}_{I\Gamma}\\
{\A}_{\Gamma I} &{\A}_{\Gamma\Gamma}
\end{pmatrix}
\begin{pmatrix}
\lambdae_I\\
\lambdae_\Gamma
\end{pmatrix}=\begin{pmatrix}
{\bf b}_I\\
{\bf b}_\Gamma
\end{pmatrix}.
\end{align*}
Therefore, for each subdomain $\Omega_i$, the subdomain problem can be written as 
\begin{align}\label{Asubdomain}
\begin{pmatrix}
{\A}^{(i)}_{II}& {\A}^{(i)}_{I\Gamma}\\
{\A}^{(i)}_{\Gamma I} &{\A}^{(i)}_{\Gamma\Gamma}
\end{pmatrix}
\begin{pmatrix}
\lambdae^{(i)}_I\\
\lambdae^{(i)}_\Gamma
\end{pmatrix}=\begin{pmatrix}
{\bf b}^{(i)}_I\\
{\bf b}^{(i)}_\Gamma
\end{pmatrix}.
\end{align}

By \eqref{Asubdomain}, we can define the subdomain local Schur complement ${\bf S}^{(i)}_\Gamma$ as follows:
\begin{align}\label{eq:schurSi}
{\S}^{(i)}_\Gamma\lambdae^{(i)}_\Gamma={\g}^{(i)}_\Gamma,
\end{align}
where
\begin{align*}
{\S}^{(i)}_\Gamma={\A}^{(i)}_{\Gamma\Gamma}-{\A}^{(i)}_{\Gamma I}{\A}^{(i)^{-1}}_{II}{\A}^{(i)}_{I\Gamma},\quad {\g}^{(i)}_\Gamma={\bf b}^{(i)}_\Gamma-{\A}^{(i)}_{\Gamma I}{\A}^{(i)^{-1}}_{II}{\bf b}^{(i)}_I.
\end{align*}
Denote $\R^{(i)}_\Gamma$ as the restriction operator from $\widehat{\Lambdae}_\Gamma$ to $\Lambdae^{(i)}_\Gamma,$ where $\Lambdae^{(i)}_\Gamma$ is the subdomain local interface space. Assemble the subdomain local Schur complement ${\S}^{(i)}_\Gamma$, we can obtain the global Schur interface problem: find  $\lambdae_\Gamma\in\widehat{\Lambdae}_{\Gamma}$ such that 
\begin{align} \label{Ginterface}
\widehat{\S}_\Gamma\lambdae_\Gamma &= {\bf g}_\Gamma,
\end{align}
where 
\begin{align*}
\widehat{\S}_\Gamma=\sum\limits_{i=1}^N \R^{(i)^T}_{\Gamma}{\S}^{(i)}_\Gamma \R^{(i)}_\Gamma,\qquad {\g}_\Gamma=\sum\limits_{i=1}^N{\R}^{(i)^T}_\Gamma{\g}^{(i)}_\Gamma.
\end{align*}
Here $\widehat{\S}_\Gamma$ is the global Schur complement defined on $\widehat{\Lambdae}_\Gamma$.  
\subsection{A BDDC preconditioner}
We decompose our domain into quadrilateral subdomains. 
Let $\widehat{\bm\Lambda}_{\Pi}:=\{\lambdae_\Pi\}$ be the coarse level, primal interface space which is continuous across the subdomain interface. The remaining subdomain degrees of freedom are denoted as ${\bm\Lambda}_{\Delta}:=\{\lambdae_{\Delta}\}$ which are discontinuous across the subdomain interface. Moreover, the space ${\bm\Lambda}_\Delta$  can be written as the direct sum of ${\bm\Lambda}^{(i)}_{\Delta}$ which are discontinuous across the subdomain interface and have a mean value on the subdomain edge/face.

We introduce a partially assembled interface space $\widetilde{\bm\Lambda}_\Gamma$ defined as
\begin{align*}
\widetilde{\bm\Lambda}_\Gamma=\widehat{\bm\Lambda}_{\Pi}\bigoplus{\bm\Lambda}_{\Delta}=
\widehat{\bm\Lambda}_{\Pi}\bigoplus(\prod\limits_{i=1}^N{\bm\Lambda}^{(i)}_{\Delta}).
\end{align*}

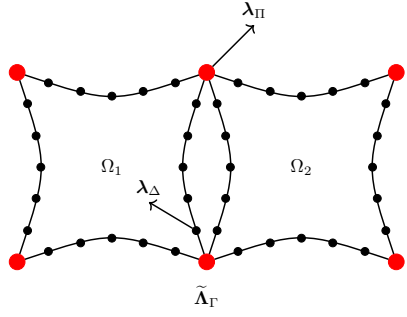
\begin{figure}[h!]
\centering
\scalebox{0.7}{
\begin{tikzpicture}[scale=1.8, node distance=0.67cm]

\tikzstyle{node} = [circle, draw=black, fill=black, inner sep=1.5pt]
\tikzstyle{special node} = [circle, draw=red, fill=red, inner sep=3pt]

\node[special node] (A) at (0, 2) {};
\node[node] (p1) at (0.33, 1.89) {};
\node[node] (p2) at (0.67, 1.8) {};
\node[node] (p3) at (1, 1.75){};
\node[node] (p4) at (1.33, 1.8){};
\node[node] (p5) at (1.67, 1.89){};
\node[special node] (D) at (2, 2) {};

\node[node] (p6) at (2.33, 1.89) {};
\node[node] (p7) at (2.67, 1.8) {};
\node[node] (p8) at (3, 1.75){};
\node[node] (p9) at (3.33, 1.8){};
\node[node] (p10) at (3.67, 1.89){};
\node[special node] (E) at (4, 2) {};

\node[node] (p11) at (0.11, 1.67) {};
\node[node] (p12) at (0.2, 1.33) {};
\node[node] (p13) at (0.25, 1){};
\node[node] (p14) at (0.2, 0.67){};
\node[node] (p15) at (0.11, 0.33){};
\node[special node] (I) at (0, 0) {};

\node[node] (p16) at (0.33, 0.11) {};
\node[node] (p17) at (0.67, 0.2) {};
\node[node] (p18) at (1, 0.25) {};
\node[node] (p19) at (1.33, 0.2) {};
\node[node] (p20) at (1.67, 0.11){};
\node[special node] (L) at (2, 0) {};

\node[node] (p21) at (2.33, 0.11) {};
\node[node] (p22) at (2.67, 0.2) {};
\node[node] (p23) at (3, 0.25) {};
\node[node] (p24) at (3.33, 0.2) {};
\node[node] (p25) at (3.67, 0.11){};
\node[special node] (M) at (4, 0) {};

\node[node] (p26) at (3.89, 1.67) {};
\node[node] (p27) at (3.8, 1.33) {};
\node[node] (p28) at (3.75, 1){};
\node[node] (p29) at (3.8, 0.67) {};
\node[node] (p30) at (3.89, 0.33){};

\node[node] (p31) at (1.75, 1) {};
\node[node] (p32) at (2.25, 1) {};

\node[node] (p33) at (1.8, 1.33) {};
\node[node] (p34) at (2.2, 1.33) {};

\node[node] (p35) at (1.89, 1.67) {};
\node[node] (p36) at (2.11, 1.67) {};

\node[node] (p37) at (1.8, 0.67) {};
\node[node] (p38) at (2.2, 0.67) {};

\node[node] (p39) at (1.89, 0.33) {};
\node[node] (p40) at (2.11, 0.33) {};

\draw[thick] (D) .. controls (1.67, 1) .. (L);
\draw[thick] (D) .. controls (2.33, 1) .. (L);

\draw[thick] (A) .. controls (0.33, 1) .. (I);
\draw[thick] (E) .. controls (3.67, 1) .. (M);

\draw[thick] (A) .. controls (1, 1.67) .. (D);
\draw[thick] (I) .. controls (1, 0.33) .. (L);

\draw[thick] (D) .. controls (3, 1.67) .. (E);
\draw[thick] (L) .. controls (3, 0.33) .. (M);

\node at (1, 1) {\(\Omega_1\)};
\node at (3, 1) {\(\Omega_2\)};

\draw[->, thick] (D) -- ++(0.5, 0.5) node[anchor=south] {\(\lambdae_{\Pi}\)};
\draw[->, thick] (p39) -- ++(-0.5, 0.3) node[anchor=south] {\(\lambdae_{\Delta}\)};

\node[below=0.33cm] at (2, 0) {$\widetilde\Lambdae_\Gamma$};

\end{tikzpicture}}
\caption{The space $\widetilde{\bm{\Lambda}}_\Gamma$}
\end{figure}

In order to introduce the BDDC preconditioner, we first introduce several operators. Let $\R^{(i)}_{\Delta}$  be the map from space $\widehat{\Lambdae}_\Gamma$ to $\Lambdae^{(i)}_{\Delta}$; $\overline{\bm R}^{(i)}_\Gamma$ be the restriction operator from $\widetilde{\Lambdae}_\Gamma$ to $\Lambdae^{(i)}_\Gamma$ and 
$\overline{\bm R}_\Gamma$ is the direct sum of $\R^{(i)}_{\Gamma}.$
Then  we can define a scaling factor $\delta^{\dagger}_i(x)$. Let ${\bm D}^{(i)}$ be the diagonal matrix with $\delta^{\dagger}_i(x)$ on its diagonal. Multiply ${\bm D}^{(i)}$ by $\R^{(i)}_{\Delta}$, we obtain the operator $\R^{(i)}_{\bm D,\Delta}$. Let
$\R_{\Gamma\Pi}$ be the map from $\widehat{\Lambdae}_\Gamma$ to $\widehat{\Lambdae}_{\Pi}$, then 
$\widetilde{\R}_{\D,\Gamma}$ can be defined as
the direct sum of $\R_{\Gamma\Pi}$ and $\R^{(i)}_{D,\Delta}$.
There are multiple scaling factor options \cite{Widlund:2020:BDDC,ZT:2016:Darcy} as long as the following conditions are obtained
\begin{align*}
\widetilde{\R}^T_{\bm D,\Gamma}\widetilde{\R}_{\Gamma}=\widetilde{\R}^T_{\Gamma}\widetilde{\R}_{\bm D,\Gamma}=I.
\end{align*}

Here we choose the simple scaling factor defined as
$$\delta^{\dagger}_i(x)=\frac1{card(I_x)},x\in \partial\Omega_{i,h}\cap\Gamma_h,$$
where $I_x$ is the indices set of subdomains which node $x$ belongs to and $card(I_x)$ is the counting number of such subdomains in $I_x$. 

Let ${\S}_\Gamma$ be the direct sum of the subdomain local Schur complement  $\S^{(i)}_\Gamma$,
the partially assembled interface Schur complement is defined as
\begin{align}\label{Sgamma}
\widetilde{\S}_\Gamma=\overline{\R}_\Gamma^T{\S}_\Gamma\overline{\R}_\Gamma.
\end{align}
With  $\widetilde{\S}_\Gamma$ defined in \eqref{Sgamma}, we define the BDDC preconditioner as  
\begin{equation}\label{BDDC}
\bm{M}^{-1}=\widetilde{\R}^T_{\bm D,\Gamma}\widetilde{\S}^{-1}_\Gamma
\widetilde{\R}_{\bm D,\Gamma}.
\end{equation}
Apply the BDDC preconditioner to the global interface problem \eqref{Ginterface}, we have
\begin{align}\label{bddc}
\widetilde{\R}^T_{\bm D,\Gamma}\widetilde{\S}^{-1}_\Gamma\widetilde{\R}_{\bm D,\Gamma}\widehat{\S}_\Gamma\lambdae_\Gamma=\widetilde{\R}^T_{\bm D,\Gamma}\widetilde{\S}^{-1}_\Gamma\widetilde{\R}_{\bm D,\Gamma}{\g}_\Gamma.
\end{align}
Since the system \eqref{Ginterface} is non-symmetric, we use GMRES to solve \eqref{bddc}. For each iteration, it will require solving the subdomain Dirichlet boundary value problem, subdomain Robin boundary value problem, and a coarse level problem (cf. \cite{Li:2004:FBC}). 

 For the BDDC preconditioner, we employ the subdomain edge average constraints in $\widehat{\bm \Lambda}_{\Pi}$ as the coarse level primal constrain such that for two adjacent subdomains $\Omega_i$ and $\Omega_j$ that shares the same edge $\mathcal{E}_{ij}$
\begin{equation}\label{constrain}
\int_{\mathcal{E}_{ij}} \lambdae^{(i)}_\Gamma ds
\end{equation}
are the same. 

{\begin{remark}
Edge average constraints are used as the primal constraints in the BDDC algorithm (cf. \cite{Li:2004:FBC}) in this paper, ensuring continuity across the subdomain vertices. Additional constraints can be introduced to accelerate the convergence of the left preconditioned GMRES algorithm and improve the performance (see \cite{TZAD2021} and \cite{TuLi:2008:BDDCAD}). 
\end{remark}}

\subsection{Overall BDDC Algorithm}

In our preconditioned BDDC algorithm, we first obtain $\lambdae_\Gamma$ by solving the global interface problem \eqref{bddc}, then we recover $\lambdae^{(i)}_I$ inside each subdomain $\Omega_i$ by \eqref{Asubdomain}.
After recovering $\lambdae^{(i)}_I$, we can obtain $\lambdae^{(i)}$ on the edge of each element in subdomain $\Omega_i$. Finally, by multiplying $-\beta^{\frac12}$ to equations \eqref{eq:qK} and \eqref{eq:pK}, together with equations \eqref{eq:yK} and \eqref{eq:pKu}, we obtain $(\bm{L}^{(i)},\bm{u}^{(i)})$ as the solution to the subdomain local problem
\begin{align}
\begin{pmatrix}
\A^{(i)}_{\L\L} &{\A^{(i)}_{\u\L}}^T & {\A^{(i)}_{\lambdae\L}}^T\\
{\A^{(i)}_{\u\L}} &\A^{(i)}_{\u\u} &\A^{(i)}_{\u\lambdae}\\
\end{pmatrix}
\begin{pmatrix}
\L^{(i)}\\
\u^{(i)}\\
\lambdae^{(i)}
\end{pmatrix}=\begin{pmatrix}
\bm{0}\\
\bm{Z}^{(i)}\\
\end{pmatrix},
\end{align}
where the matrices and vectors with superscript $(i)$ are obtained by restricting those matrices and vectors in the subdomain $\Omega_i$. A detailed description of the BDDC preconditioner is provided in Algorithm \ref{algo:bddchdg}.

\begin{algorithm}
\caption{BDDC preconditioned Algorithm \eqref{bddc}}\label{algo:bddchdg}
\begin{algorithmic}[1]
\FOR{each subdomain $\Omega_i$}
    \STATE Generate $\A^{(i)}_{\alpha\beta}$, where $\A^{(i)}_{\alpha\beta}$ are the restrictions of matrices appearing in \eqref{Aqu} in subdomain $\Omega_i$. 
 \STATE Generate  $\A^{(i)}$, where 
    ${\scriptstyle \A^{(i)}=\A^{(i)}_{\lambdae\lambdae}-}\begin{pmatrix}
\A^{(i)}_{\lambdae\L}&\A^{(i)}_{\lambdae\u}
\end{pmatrix}\begin{pmatrix}
\A^{(i)}_{\L\L}&{\A^{(i)}_{\u\L}}^T\\
{\A^{(i)}_{\u\L}}&\A^{(i)}_{\u\u}
\end{pmatrix}^{-1}
\begin{pmatrix}
{\A^{(i)}_{\lambdae\L}}^T\\
\A^{(i)}_{\u\lambdae}
\end{pmatrix}.$
     
\STATE Employ the subdomain Dirichlet boundary condition when $\partial\Omega_i\cap\partial\Omega  \neq \emptyset$ and employ the  subdomain Robin boundary condition to make subdomain local problem solvable, i.e 
\begin{align*}
&\l\cT_1\widehat{y}_h,\mu_1\r_{\partial\cT_h(\Omega_i)}=\l\cT_1\widehat{y}_h,\mu_1\r_{\partial\cT_h(\Omega_i)}+\frac12\langle\bm{\zeta}\cdot\bn\widehat{y}_h, \mu_1\rangle_{\partial\cT_h(\Omega_i)},\\
&\l\cT_2\widehat{p}_h,\mu_2\r_{\partial\cT_h(\Omega_i)}=\l\cT_2\widehat{p}_h,\mu_2\r_{\partial\cT_h(\Omega_i)}-\frac12\langle\bm{\zeta}\cdot\bn  \widehat{p}_h,\mu_2\rangle_{\partial\cT_h(\Omega_i)}
\end{align*}
and $\A^{(i)}_{\widehat{y}\widehat{y}}, \A^{(i)}_{\widehat{p}\widehat{p}}, \A^{(i)}_{\lambdae\lambdae}$ are  modified accordingly. 
\STATE Enforce edge average constraint \eqref{constrain} on each subdomain interface $\partial\Omega_i$.
\ENDFOR
\STATE Initialize $\lambdae_\Gamma$ on the coarse mesh.
\STATE Generate the global interface problem \eqref{Ginterface}  by assembling $\S^{(i)}_\Gamma$ in \eqref{eq:schurSi} in each domain $\Omega_i$.
\STATE Generate the BDDC preconditioner \eqref{BDDC}.
\STATE Use GMRES solver with BDDC preconditioner to solve for $\lambdae_\Gamma$.
\STATE Map $\lambdae_\Gamma$ onto each subdomain $\Omega_i$ to recover $\lambdae^{(i)}_I$ and subsequently  recover $\L^{(i)},\u^{(i)}$ within each subdomain $\Omega_i$. 
\end{algorithmic}
\end{algorithm}

\section{Numerical Results}\label{sec:numerics}
In this section, we present three numerical examples in two dimensions to illustrate our theoretical results. We solve the discrete problem $\eqref{eq:hdg1}$ with $k=1$ and $2$ in $\Omega=[0,1]\times[0,1]$. 
The stabilization parameters in $\eqref{eq:hdg1}$ are defined as  $$\tau_1=\max(\sup\limits_{{x\in\mathcal{E}}}(\zeta\cdot{\bf n}),0)+1,\quad\forall
\mathcal{E}\subset K, \forall K\in\mathcal{T}_h$$ and $\tau_2=\tau_1-\bm{\zeta}\cdot\bn$ according to assumption \eqref{assump2}.

For BDDC algorithm, we decompose the domain $\Omega$ into quadrilateral subdomains with meshsize $H$, where each subdomain consists of triangles in $\cT_h$ with meshsize $h$. For all the convergence rates, we let $h=6^{-1}\cdot2^{-l-1}$ at level $l$ due to the setting $H/h=6$. We use GMRES to solve the system and consider the error between the exact solution $(\bq,y,\bp,p)$ to \eqref{Osystem} and the HDG solution $(\bq_h,\bp_h,y_h,p_h,\widehat{y}_h,\widehat{p}_h)$ to \eqref{eq:hdg1}
in the $\|\cdot\|_{1,\beta}$ norm and in the $L_2$ norm. The GMRES algorithm is stopped when the residual is reduced by $10^{-11}$. All the computation is performed in MATLAB.

\begin{example}[Constant convection]\label{ex:cc}
In this example, we take $\bm{\zeta}=[1,0]^t, \gamma=1$ in \eqref{eq:hdg1} which satisfy the assumption \eqref{eq:advassump}. Let the exact solution to \eqref{eq:fsystem} be
$${y}=\sin(\pi x)\sin(\pi y),{p}=\sin(\pi x)\sin(\pi y).$$ The right-hand sides $f$ and $g$ are calculated accordingly.
\end{example}

We first report the convergence rates of the HDG methods. In Table \ref{table:cch1}, we calculate the errors $\|(\bq-\bq_h,y-y_h,y-\widehat{y}_h)\|_{1,\beta}+\|(\bp-\bp_h,p-p_h,p-\widehat{p}_h)\|_{1,\beta}$ where the $\|\cdot\|_{1,\beta}$ norm is defined in \eqref{eq:enorm}. We observe that for $\beta=1, 10^{-2}$, the convergence are $O(h^{k+\frac12})$ for both polynomial degrees. The convergence rates tend to $O(h^{k+1})$ as $\beta$ goes to zero. These results are consistent with our result in Theorem \ref{thm:hdgesti}. 
We also report the errors $\|y-y_h\|_{\LT}$ and $\|p-p_h\|_{\LT}$ in Tables \ref{table:ccl2y} and \ref{table:ccl2p} respectively. We observe almost $O(h^{k+1})$ convergence rates for both variables, which are better than our results in Theorem \ref{thm:hdgesti}. Indeed, we do not utilize any duality argument to establish the $L_2$ estimates, rather, they follow directly as a simple consequence of Theorem \ref{thm:hdgesti}. Moreover, the convergence rates of our $L_2$ norm results are consist with those in \cite{chen2018hdg} when $\beta=1$.

\begin{table}
\footnotesize
\caption{Convergence rates for Example \ref{ex:cc} in the energy norm $\|\cdot\|_{1,\beta}$}\label{table:cch1}
\begin{center}
  \begin{tabular}{ccc|cc|cc|cc} \hline
&\multicolumn{8}{c}{$k=1$}\\
\hline
\multirow{2}{*}{$l$}&\multicolumn{2}{c}{$\beta=1$}&\multicolumn{2}{c}{$\beta=10^{-2}$}&\multicolumn{2}{c}{$\beta=10^{-4}$}&\multicolumn{2}{c}{$\beta=10^{-8}$}\\
\cline{2-9}
&Error&Rate&Error&Rate&Error&Rate&Error&Rate\\
\hline
$1$\ \vline&1.58e-2&-&6.19e-3&-&4.06e-3&-&2.59e-3&-\\
 $2$\ \vline&4.99e-3&1.66&1.82e-3&1.77&1.07e-3&1.92&6.51e-4&1.99\\
 $3$\ \vline&1.64e-3&1.61&5.68e-4&1.68&2.89e-4&1.89&1.65e-4&1.98\\
 $4$\ \vline&5.59e-4&1.55&1.86e-4&1.61&8.16e-5&1.82&4.30e-5&1.94\\
 \hline
 &\multicolumn{8}{c}{$k=2$}\\
\hline
\multirow{2}{*}{$l$}&\multicolumn{2}{c}{$\beta=1$}&\multicolumn{2}{c}{$\beta=10^{-2}$}&\multicolumn{2}{c}{$\beta=10^{-4}$}&\multicolumn{2}{c}{$\beta=10^{-8}$}\\
\cline{2-9}
&Error&Rate&Error&Rate&Error&Rate&Error&Rate\\
\hline
1\ \vline&3.50e-4&-&1.32e-4&-&8.26e-5&-&4.76e-5&-\\
2\ \vline&5.70e-5&2.62&2.01e-5&2.72&1.10e-5&2.91&6.01e-6&2.99\\
3\ \vline&9.61e-6&2.57&3.24e-6&2.63&1.52e-6&2.86&7.78e-7&2.95\\
4\ \vline&1.66e-6&2.53&5.42e-7&2.58&2.22e-7&2.78&1.07e-7&2.86\\
\hline
  \end{tabular}
\end{center}
\footnotesize
\caption{Convergence rates of $y$ for Example \ref{ex:cc} in the $L_2$ norm}\label{table:ccl2y}
\begin{center}
\begin{tabular}{ccc|cc|cc|cc}
\hline
&\multicolumn{8}{c}{$k=1$}\\
\hline
\multirow{2}{*}{$h$}&\multicolumn{2}{c}{$\beta=1$}&\multicolumn{2}{c}{$\beta=10^{-2}$}&\multicolumn{2}{c}{$\beta=10^{-4}$}&\multicolumn{2}{c}{$\beta=10^{-8}$}\\
\cline{2-9}
&Error&Rate&Error&Rate&Error&Rate&Error&Rate\\
\hline
1\ \vline&2.73e-3&-&2.76e-3&-&2.97e-3&-&1.83e-3&-\\
2\ \vline&6.79e-4&2.00&6.84e-4&2.01&7.16e-4&2.05&4.59e-4&2.00\\
3\ \vline&1.70e-4&2.00&1.70e-4&2.01&1.74e-4&2.04&1.17e-4&1.97\\
4\ \vline&4.23e-5&2.01&4.24e-5&2.00&4.30e-5&2.02&3.11e-5&1.91\\
\hline
&\multicolumn{8}{c}{$k=2$}\\
\hline
\multirow{2}{*}{$h$}&\multicolumn{2}{c}{$\beta=1$}&\multicolumn{2}{c}{$\beta=10^{-2}$}&\multicolumn{2}{c}{$\beta=10^{-4}$}&\multicolumn{2}{c}{$\beta=10^{-8}$}\\
\cline{2-9}
&Error&Rate&Error&Rate&Error&Rate&Error&Rate\\
\hline
1\ \vline&5.37e-5&-&5.42e-5&- & 5.81e-5& -&3.36e-5&- \\
2\ \vline&6.71e-6&3.00&6.74e-6&3.01 &7.01e-6 & 3.05   &4.24e-6&2.99\\
3\ \vline&8.39e-7&3.00&8.40e-7&3.00   &8.58e-7 &3.03  &5.55e-7&2.93\\
4\ \vline&1.05e-7&3.00&1.05e-7&3.00    &1.06e-7 &3.02  &8.16e-8&2.77\\
\hline
  \end{tabular}
\end{center}

\footnotesize
\caption{Convergence rates of $p$ for Example \ref{ex:cc} in the $L_2$ norm}\label{table:ccl2p}
\begin{center}
\begin{tabular}{ccc|cc|cc|cc}
\hline
&\multicolumn{8}{c}{$k=1$}\\
\hline
\multirow{2}{*}{$h$}&\multicolumn{2}{c}{$\beta=1$}&\multicolumn{2}{c}{$\beta=10^{-2}$}&\multicolumn{2}{c}{$\beta=10^{-4}$}&\multicolumn{2}{c}{$\beta=10^{-8}$}\\
\cline{2-9}
&Error&Rate&Error&Rate&Error&Rate&Error&Rate\\
\hline
1\ \vline&2.72e-3&-&2.67e-3&-&2.32e-3&-&1.84e-3&-\\
2\ \vline&6.78e-4&2.00&6.73e-4&1.99&6.29e-4&1.88&4.61e-4&2.00\\
3\ \vline&1.69e-4&2.00&1.69e-4&1.99&1.63e-4&1.95&1.16e-4&1.99\\
4\ \vline&4.23e-5&2.00&4.23e-5&2.00&4.16e-5&1.97&2.96e-5&1.97\\
\hline
&\multicolumn{8}{c}{$k=2$}\\
\hline
\multirow{2}{*}{$h$}&\multicolumn{2}{c}{$\beta=1$}&\multicolumn{2}{c}{$\beta=10^{-2}$}&\multicolumn{2}{c}{$\beta=10^{-4}$}&\multicolumn{2}{c}{$\beta=10^{-8}$}\\
\cline{2-9}
&Error&Rate&Error&Rate&Error&Rate&Error&Rate\\
\hline
1\ \vline&5.36e-5&-&5.32e-5&-&4.80e-5&-&3.37e-5&-\\
2\ \vline&6.71e-6&3.00&6.68e-6&2.99&6.36e-6&2.92&4.26e-6&2.98\\
3\ \vline&8.38e-7&3.00&8.36e-7&3.00&8.17e-7&2.96&5.43e-7&2.97\\
4\ \vline&1.05e-7&3.00&1.05e-7&2.99&1.03e-7&2.99&6.85e-8&2.99\\
\hline
\end{tabular}
\end{center}
\end{table}

We then report the number of iterations for BDDC preconditioned GMRES. In Table \ref{table:ccgmrescounts}, we set 
$H/h=6$ and present the iteration counts for different values of $\beta$ and various numbers of subdomains. We observe that the iteration counts remain independent of the number of subdomains for $k=1$ and $k=2$ with a fixed $\beta$. This is consistent with the results in \cite{TWZstokes2020,TZAD2021,TuLi:2008:BDDCAD} and the references therein, which shows the numerical scalability of the BDDC algorithm. Additionally, we find that the number of iterations decreases as $\beta$ decreases, which is also consistent with the results in \cite{brenner2020multigrid}. In Table \ref{table:ccgmrescounts1}, we fix the number of subdomains at 36 and report the iteration counts for different values of $h$. We observe a slight increase in the number of iterations as $H/h$ increases, which aligns with the findings in \cite{TZAD2021} and \cite{TWZstokes2020}. Additionally, we clearly see that the BDDC algorithm is robust with respect to $\beta$, requiring fewer iterations for smaller $\beta$, which is again consistent with the results in \cite{brenner2020multigrid}.

\begin{table}
\footnotesize
\caption{GMRES numbers of iterations for Example \ref{ex:cc} with different values of $\beta$ and $H/h=6$}\label{table:ccgmrescounts}
\begin{center}
\begin{tabular}{ccccc|cccc}
&\multicolumn{4}{c}{$k=1$}&\multicolumn{4}{c}{$k=2$}\\
\hline
\multirow{2}{*}{$\beta$}&\multicolumn{8}{c}{Number of subdomains}\\
\cline{2-9}\\[-2ex]
&$4^2$&$8^2$&$16^2$&$32^2$&$4^2$&$8^2$&$16^2$&$32^2$\\
\hline
$1$&21  & 25  &25   & 24 &29  &29  &30   &35\\
$10^{-2}$&22  &25  &23   &21  & 30  &24  &25  &30\\
$10^{-4}$&25  &24  &23  &21  &33  &24  &28   &30\\
$10^{-6}$&17  &24  &19   &19  &21  &27  &25   &21\\
$10^{-8}$&8  &11  &17   &20 &10  &15  &21   &26\\
$10^{-10}$&6  & 7 &8   &11 &7  &8  &11   &15\\
\hline
\end{tabular}
\end{center}
\footnotesize
\caption{GMRES numbers of iterations for Example \ref{ex:cc} with different values of $\beta$ and $6\times 6$ subdomains}\label{table:ccgmrescounts1}
\begin{center}
\begin{tabular}{cllll|llll}
&\multicolumn{4}{c}{$k=1$}&\multicolumn{4}{c}{$k=2$}\\
\hline
\multirow{2}{*}{$\beta$}&\multicolumn{8}{c}{$H/h$}\\
\cline{2-9}\\[-2ex]
&$4$&$8$&$16$&$20$&$4$&$8$&$16$&$20$\\
\hline
$1$  &21  & 27  &29   & 28 &28  &33  &39   &42 \\
$10^{-2}$ &21  &28  &26   &28  & 24  &28  &35  &38 \\    
$10^{-4}$ &21  &26  &30   &30  &24  &28  &33   &36 \\   
$10^{-6}$ &18  &23  &31   &34  &22  &25  &30   &33 \\ 
$10^{-8}$ &8  &11  &15   &17 &11  &15  &21   &23  \\ 
$10^{-10}$&6  & 7 &8   &9 &7  &8  &10   &12  \\ 
\hline
\end{tabular}
\end{center}
\footnotesize
\caption{Convergence rates for Example \ref{ex:vc} in the energy norm $\|\cdot\|_{1,\beta}$}\label{table:vch1}
\begin{center}
\begin{tabular}{ccc|cc|cc|cc}
\hline
&\multicolumn{8}{c}{$k=1$}\\
\hline
\multirow{2}{*}{$l$}&\multicolumn{2}{c}{$\beta=1$}&\multicolumn{2}{c}{$\beta=10^{-2}$}&\multicolumn{2}{c}{$\beta=10^{-4}$}&\multicolumn{2}{c}{$\beta=10^{-8}$}\\
\cline{2-9}
&Error&Rate&Error&Rate&Error&Rate&Error&Rate\\
\hline
$1$\ \vline&1.52e-2&-&6.07e-3&-&4.09e-3&-&2.59e-3&-\\
$2$\ \vline&4.79e-3&1.66&1.78e-3&1.77&1.08e-3&1.92&6.51e-4&1.99\\
$3$\ \vline&1.58e-3&1.60&5.52e-4&1.69&2.90e-4&1.90&1.65e-4&1.98\\
$4$\ \vline&5.39e-4&1.55&1.80e-4&1.62&8.14e-5&1.83&4.30e-5&1.94\\
\hline
&\multicolumn{8}{c}{$k=2$}\\
\hline
\multirow{2}{*}{$l$}&\multicolumn{2}{c}{$\beta=1$}&\multicolumn{2}{c}{$\beta=10^{-2}$}&\multicolumn{2}{c}{$\beta=10^{-4}$}&\multicolumn{2}{c}{$\beta=10^{-8}$}\\
\cline{2-9}
&Error&Rate&Error&Rate&Error&Rate&Error&Rate\\
\hline
1\ \vline&4.67e-4&-&1.28e-4&-&8.35e-5&-&4.76e-5&-\\
2\ \vline&7.59e-5&2.62&1.93e-5&2.73&1.11e-5&2.91&6.01e-6&2.99\\
3\ \vline&1.28e-5&2.57&3.09e-6&2.64&1.52e-6&2.87&7.78e-7&2.95\\
4\ \vline&2.21e-6&2.53&5.15e-7&2.59&2.18e-7&2.80&1.07e-7&2.86\\
\hline
\end{tabular}
\end{center}
\end{table}

\begin{example}[Variable convection]\label{ex:vc}
In this example, we take $\bm{\zeta}=[y,-x]^t$ and $\gamma=1$ in \eqref{eq:hdg1}, which satisfy \eqref{eq:advassump}. Let the exact solution be ${y}=\sin(\pi x)\sin(\pi y),{p}=\sin(\pi x)\sin(\pi y)$.  
\end{example}

Similarly to Example \ref{ex:cc}, we report the errors in $\|\cdot\|_{1,\beta}$ norm in Example \ref{ex:vc} with $k=1$ and $k=2$. We again observe $O(h^{k+\frac12})$ for $\beta=1, 10^{-2}$ and almost $O(h^{k+1})$ as $\beta$ goes to zero. We can also see similar convergence rates in $L_2$ norm for $y$ and $p$ as those in Example \ref{ex:cc} in Tables \ref{table:vcl2y} and \ref{table:vcl2p}.

We also report the number of iterations of GMRES in Tables \ref{table:vcgmrescounts} and \ref{table:vcgmrescounts1} similarly to those in Example \ref{ex:cc}. Again, we observe that our BDDC preconditioner is robust with respect to $\beta$.

\begin{table}
\footnotesize
\caption{Convergence rates of $y$ for Example \ref{ex:vc} in the $L_2$ norm}\label{table:vcl2y}
\begin{center}
\begin{tabular}{ccc|cc|cc|cc}
\hline
&\multicolumn{8}{c}{$k=1$}\\
\hline
\multirow{2}{*}{$h$}&\multicolumn{2}{c}{$\beta=1$}&\multicolumn{2}{c}{$\beta=10^{-2}$}&\multicolumn{2}{c}{$\beta=10^{-4}$}&\multicolumn{2}{c}{$\beta=10^{-8}$}\\
\cline{2-9}
&Error&Rate&Error&Rate&Error&Rate&Error&Rate\\
\hline
1\ \vline&2.78e-3&-&2.81e-3&-&3.04e-3&-&1.83e-3&-\\
2\ \vline&6.94e-4&2.00&6.99e-4&2.01&7.35e-4&2.05&4.59e-4&2.00\\
3\ \vline&1.73e-4&2.00&1.74e-4&2.01&1.79e-4&2.04&1.17e-4&1.97\\
4\ \vline&4.34e-5&2.00&4.34e-5&2.00&4.41e-5&2.02&3.11e-5&1.91\\
\hline
&\multicolumn{8}{c}{$k=2$}\\
\hline
\multirow{2}{*}{$h$}&\multicolumn{2}{c}{$\beta=1$}&\multicolumn{2}{c}{$\beta=10^{-2}$}&\multicolumn{2}{c}{$\beta=10^{-4}$}&\multicolumn{2}{c}{$\beta=10^{-8}$}\\
\cline{2-9}
&Error&Rate&Error&Rate&Error&Rate&Error&Rate\\
\hline
1\ \vline&5.52e-5&-&5.57e-5&-&6.00e-5&-&3.36e-5&-\\
2\ \vline&6.92e-6&3.00&6.95e-6&3.00&7.26e-6&3.05&4.24e-6&2.99\\
3\ \vline&8.66e-7&3.00&8.68e-7&3.00&8.88e-7&3.03&5.55e-7&2.93\\
4\ \vline&1.08e-7&3.00&1.08e-7&3.01&1.10e-7&3.01&8.16e-8&2.77\\
\hline
\end{tabular}
\end{center}
\footnotesize
\caption{Convergence rates of $p$ for Example \ref{ex:vc} in the $L_2$ norm}\label{table:vcl2p}
\begin{center}
\begin{tabular}{ccc|cc|cc|cc}
\hline
&\multicolumn{8}{c}{$k=1$}\\
\hline
\multirow{2}{*}{$h$}&\multicolumn{2}{c}{$\beta=1$}&\multicolumn{2}{c}{$\beta=10^{-2}$}&\multicolumn{2}{c}{$\beta=10^{-4}$}&\multicolumn{2}{c}{$\beta=10^{-8}$}\\
\cline{2-9}
&Error&Rate&Error&Rate&Error&Rate&Error&Rate\\
\hline
1\ \vline&2.77e-3&-&2.72e-3&-&2.33e-3&-&1.84e-3&-\\
2\ \vline&6.93e-4&2.00&6.87e-4&1.99&6.38e-4&1.87&4.61e-4&2.00\\
3\ \vline&1.73e-4&2.00&1.73e-4&1.99&1.67e-4&1.93&1.16e-4&1.99\\
4\ \vline&4.33e-5&2.00&4.33e-5&2.00&4.25e-5&1.97&2.96e-5&1.97\\
\hline
&\multicolumn{8}{c}{$k=2$}\\
\hline
\multirow{2}{*}{$h$}&\multicolumn{2}{c}{$\beta=1$}&\multicolumn{2}{c}{$\beta=10^{-2}$}&\multicolumn{2}{c}{$\beta=10^{-4}$}&\multicolumn{2}{c}{$\beta=10^{-8}$}\\
\cline{2-9}
&Error&Rate&Error&Rate&Error&Rate&Error&Rate\\
\hline
1\ \vline&5.51e-5&-&5.45e-5&-&4.87e-5&-&3.37e-5&-\\
2\ \vline&6.91e-6&3.00&6.88e-6&2.99&6.52e-6&2.90&4.26e-6&2.98\\
3\ \vline&8.65e-7&3.00&8.63e-7&2.99&8.41e-7&2.95&5.43e-7&2.97\\
4\ \vline&1.08e-7&3.00&1.08e-7&3.00&1.07e-7&2.97&6.88e-8&2.98\\
\hline
\end{tabular}
\end{center}
\footnotesize
\caption{GMRES numbers of iterations for Example \ref{ex:vc} with different values of $\beta$ and $H/h=6$}\label{table:vcgmrescounts}
\begin{center}
\begin{tabular}{ccccc|cccc}
&\multicolumn{4}{c}{$k=1$}&\multicolumn{4}{c}{$k=2$}\\
\hline
\multirow{2}{*}{$\beta$}&\multicolumn{8}{c}{Number of subdomains}\\
\cline{2-9}\\[-2ex]
&$4^2$&$8^2$&$16^2$&$32^2$&$4^2$&$8^2$&$16^2$&$32^2$\\
\hline
$1$  &21  & 25  &25   &24  &29  &29  &30   &35 \\
$10^{-2}$ &22 &25  &23   &21  &30   &24  &25  &31 \\    
$10^{-4}$ &25  &24  &23   &21  &33  &24  &28   &30 \\   
$10^{-6}$ &16  &24  &19   &19  &21  &27  &25   &21 \\ 
$10^{-8}$ &8  &11  &17   &20 &10  &15  &21   &26  \\ 
$10^{-10}$&6  &7  &8  &11 &7  &8  & 11  & 15 \\ 
\hline
\end{tabular}
\end{center}
\footnotesize
\caption{GMRES numbers of iterations for Example \ref{ex:vc} with different values of $\beta$ and $6\times 6$ subdomains}\label{table:vcgmrescounts1}
\begin{center}
\begin{tabular}{cllll|llll}
&\multicolumn{4}{c}{$k=1$}&\multicolumn{4}{c}{$k=2$}\\
\hline
\multirow{2}{*}{$\beta$}&\multicolumn{8}{c}{$H/h$}\\
\cline{2-9}\\[-2ex]
&$4$&$8$&$16$&$20$&$4$&$8$&$16$&$20$\\
\hline
$1$  &21  & 27  &29   &30  &28  &32  &39   &42 \\
$10^{-2}$ &21  &28     &26   &28  &25   &28  &35  &38 \\    
$10^{-4}$ &21  &26  &30   &30  &24  &28  &33  &36 \\   
$10^{-6}$ &18  &24  &31   &34  &22  &25  &30   &33 \\ 
$10^{-8}$ &8  &11  &15   &17 &11  &15  &21   &23  \\ 
$10^{-10}$&6  &7  &8  &9 &7  &8  & 11  & 12 \\ 
\hline
\end{tabular}
\end{center}
\end{table}


\begin{example}\label{ex:bdlayer}
In this example, we take $\bm{\zeta}=[0,0]^t, \gamma=1$ in \eqref{eq:hdg1} which satisfy the assumption \eqref{eq:advassump}. We let $f=1$ and $g=0$. The exact solution of this example can be found using double sine series (cf. \cite{brenner2020multigrid}). The solutions of this example exhibit boundary layers when $\beta$ goes to zero.
\end{example}

See Figure \ref{fig:1e-4yh} for the numerical solution $y_h$  and Figure \ref{fig:1e-4y} for the exact solution $y$ when $\beta = 10^{-4}$, as well as Figures \ref{fig:1e-6yh} and \ref{fig:1e-6y} when $\beta = 10^{-6}$. One can clearly observe that the solutions $y$ indeed exhibit boundary layers as $\beta$ approaches zero.

\begin{figure}[h]
    \centering
    \begin{minipage}[b]{0.4\textwidth}
        \centering
        \includegraphics[width=\textwidth]{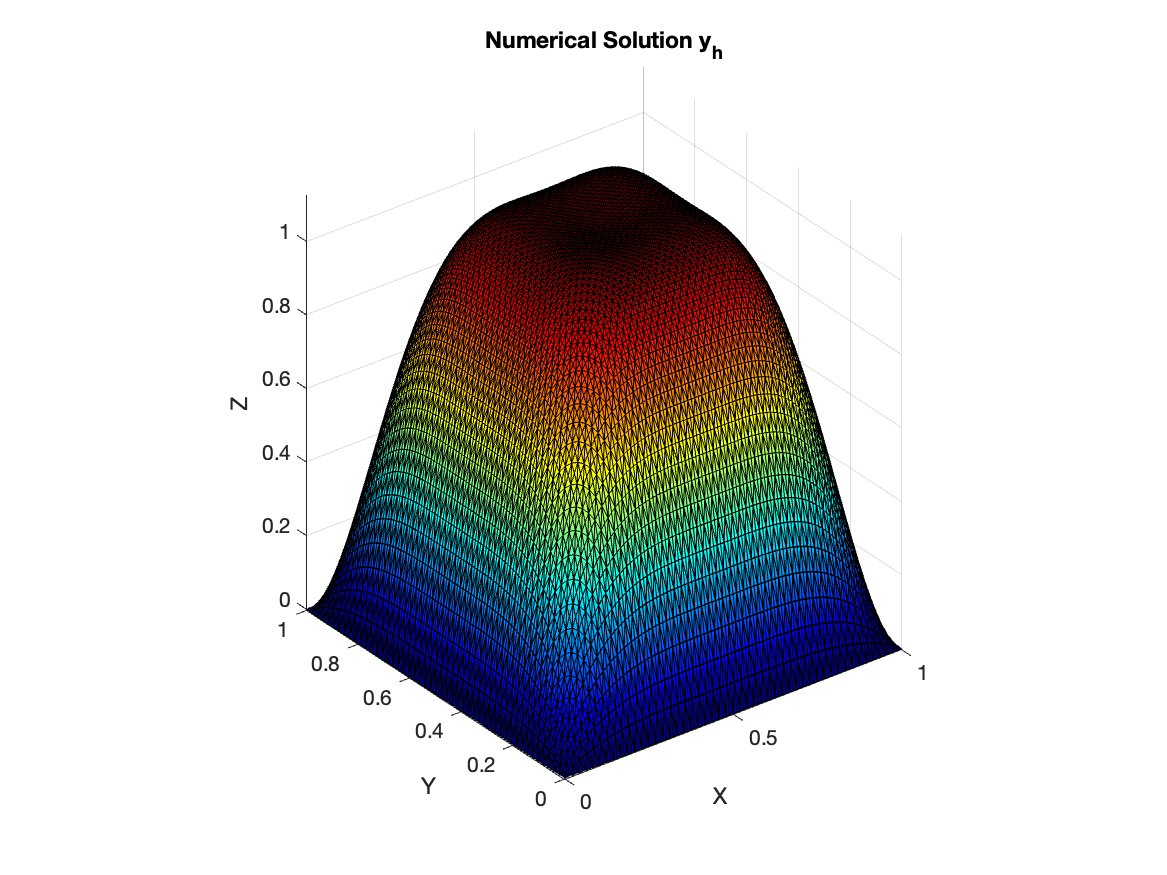}
        \caption{Numerical solution with $\beta=10^{-4}$ and $h=1/96$.}
        \label{fig:1e-4yh}
    \end{minipage}
    \hspace{1cm}
    \begin{minipage}[b]{0.4\textwidth}
        \centering
        \includegraphics[width=\textwidth]{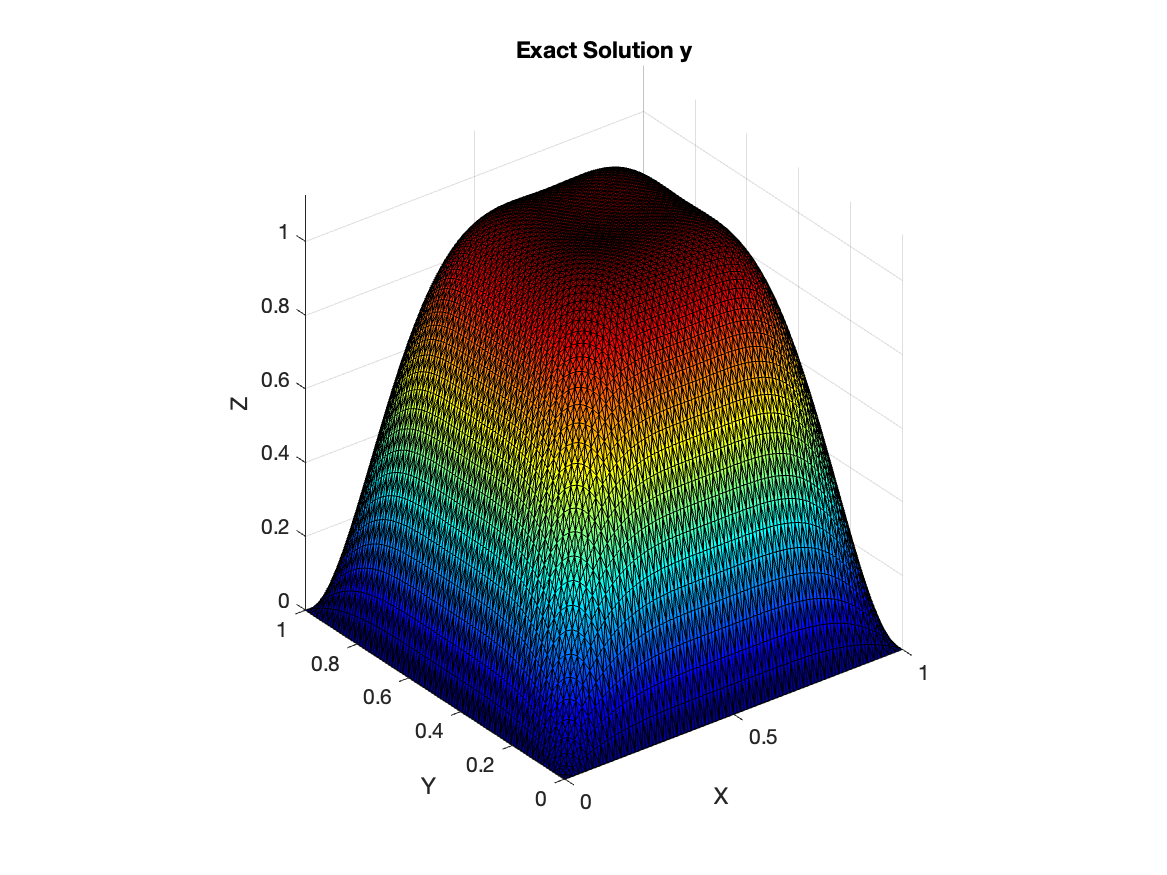}
        \caption{Exact solution with $\beta=10^{-4}$ and $h=1/96$.}
        \label{fig:1e-4y}
    \end{minipage}
\vfill
   \begin{minipage}[b]{0.4\textwidth}
        \centering
        \includegraphics[width=\textwidth]{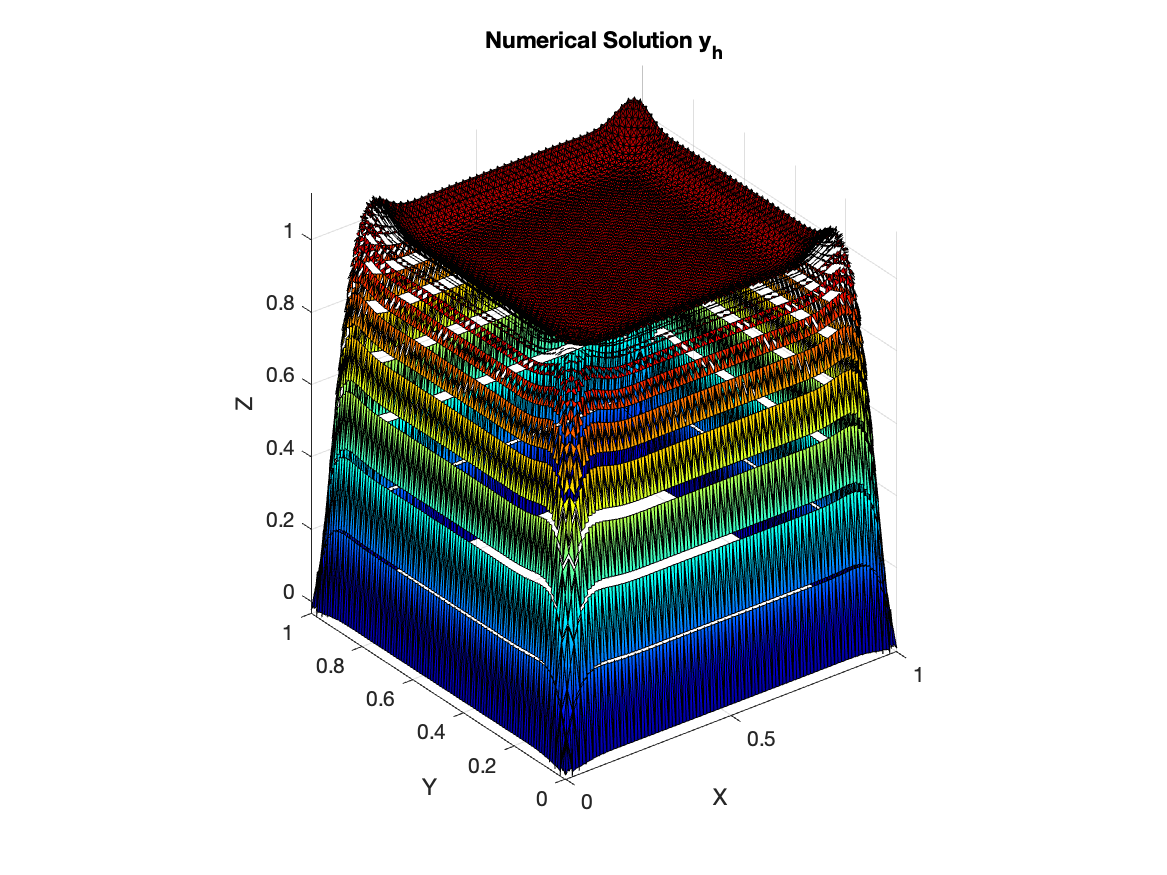}
        \caption{Numerical solution with $\beta=10^{-6}$ and $h=1/96$.}
        \label{fig:1e-6yh}
    \end{minipage}
     \hspace{1cm}
    \begin{minipage}[b]{0.4\textwidth}
        \centering
        \includegraphics[width=\textwidth]{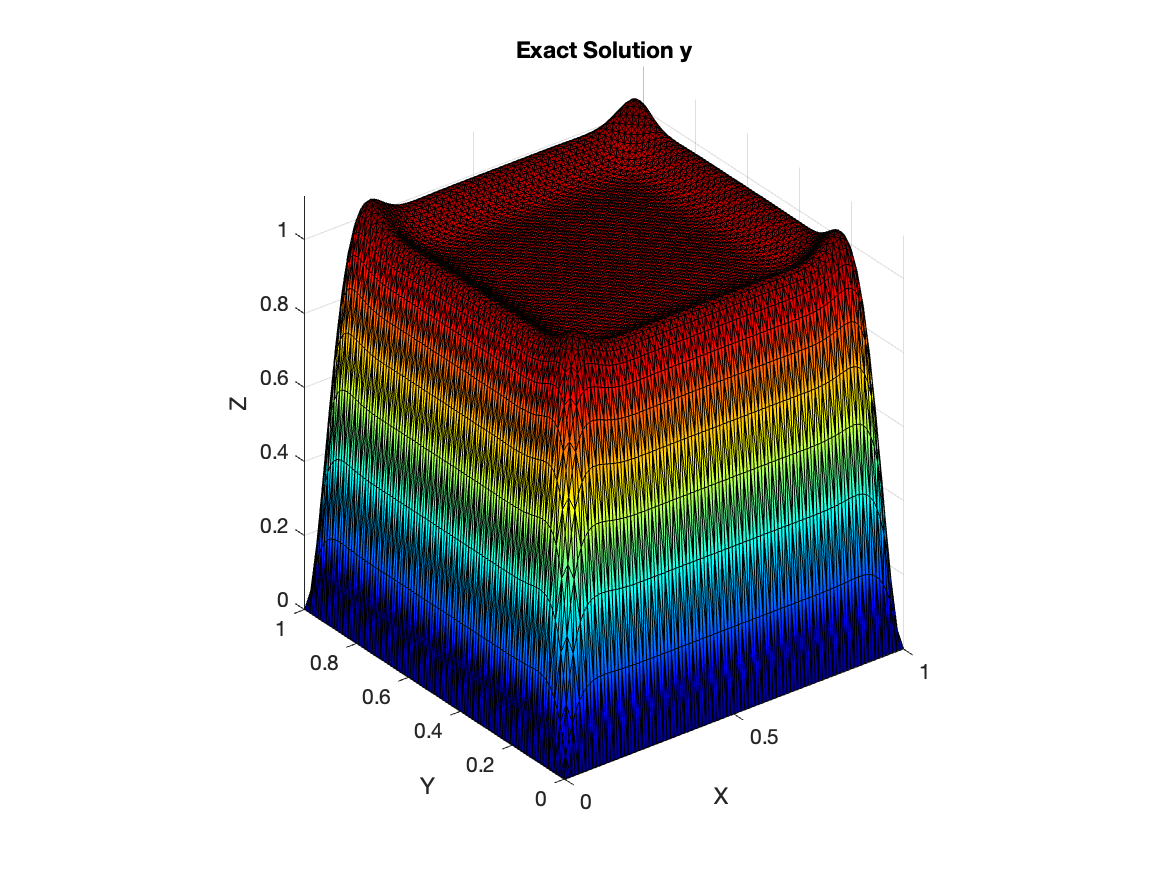}
        \caption{Exact solution with $\beta=10^{-6}$ and $h=1/96$.}
        \label{fig:1e-6y}
    \end{minipage}
    
\end{figure}

 We first report the errors and convergence rates in the $\|\cdot\|_{1, \beta}$ norm in Table \ref{table:bdh1}, where we observe that the error convergence rates are nearly $O(h^{k+1})$. Notably, the convergence rate deteriorates at coarse levels when $\beta$ is small. This is due to the presence of the boundary layer; however, the convergence rate improves as $h$ decreases, aligning with the results in Theorem \ref{thm:hdgesti}. Similar trends can be seen in Tables \ref{table:bdl2y} and \ref{table:bdl2p} for $y$ and $p$ in the $L_2$ norm. Finally, in Table \ref{table:bdgmrescounts}, we observe that the number of iterations remains independent of the number of subdomains and is robust with respect to $\beta$.

\begin{table}
\footnotesize
\caption{Convergence rates for Example \ref{ex:bdlayer} in the energy norm $\|\cdot\|_{1,\beta}$}\label{table:bdh1}
\begin{center}
\begin{tabular}{ccc|cc|cc|cc}
\hline
&\multicolumn{8}{c}{$k=1$}\\
\hline
\multirow{2}{*}{$l$}&\multicolumn{2}{c}{$\beta=1$}&\multicolumn{2}{c}{$\beta=10^{-2}$}&\multicolumn{2}{c}{$\beta=10^{-4}$}&\multicolumn{2}{c}{$\beta=10^{-6}$}\\
\cline{2-9}
&Error&Rate&Error&Rate&Error&Rate&Error&Rate\\
\hline
$1$\ \vline&8.24e-4&-&1.31e-3&-&1.23e-2&-&1.25e-1&-\\
$2$\ \vline&2.20e-4&1.91&3.33e-4&1.98&3.67e-3&1.74&4.49e-2&1.49\\
$3$\ \vline&5.83e-5&1.92&8.37e-5&1.99&9.97e-4&1.88&1.17e-2&1.94\\
$4$\ \vline&1.57e-5&1.89&2.11e-5&1.99&2.59e-4&1.94&3.12e-3&1.91\\
\hline
\end{tabular}
\end{center}
\footnotesize
\caption{Convergence rates of $y$ for Example \ref{ex:bdlayer} in the $L_2$ norm}\label{table:bdl2y}
\begin{center}
\begin{tabular}{ccc|cc|cc|cc}
\hline
&\multicolumn{8}{c}{$k=1$}\\
\hline
\multirow{2}{*}{$h$}&\multicolumn{2}{c}{$\beta=1$}&\multicolumn{2}{c}{$\beta=10^{-2}$}&\multicolumn{2}{c}{$\beta=10^{-4}$}&\multicolumn{2}{c}{$\beta=10^{-6}$}\\
\cline{2-9}
&Error&Rate&Error&Rate&Error&Rate&Error&Rate\\
\hline
1\ \vline&1.22e-5&-&9.89e-4&-&1.17e-2&-&1.24e-1&-\\
2\ \vline&3.05e-6&2.00&2.49e-4&1.99&3.50e-3&1.74&4.46e-2&1.48\\
3\ \vline&7.63e-7&2.00&6.26e-5&1.99&9.53e-4&1.88&1.17e-2&1.93\\
4\ \vline&1.91e-7&2.00&1.57e-5&2.00&2.48e-4&1.94&3.10e-3&1.92\\
5\ \vline&4.77e-8&2.00&3.92e-6&2.00&6.31e-5&1.97&9.09e-4&1.77\\
\hline
\end{tabular}
\end{center}
\footnotesize
\caption{Convergence rates of $p$ for Example \ref{ex:bdlayer} in the $L_2$ norm}\label{table:bdl2p}
\begin{center}
\begin{tabular}{ccc|cc|cc|cc}
\hline
&\multicolumn{8}{c}{$k=1$}\\
\hline
\multirow{2}{*}{$h$}&\multicolumn{2}{c}{$\beta=1$}&\multicolumn{2}{c}{$\beta=10^{-2}$}&\multicolumn{2}{c}{$\beta=10^{-4}$}&\multicolumn{2}{c}{$\beta=10^{-6}$}\\
\cline{2-9}
&Error&Rate&Error&Rate&Error&Rate&Error&Rate\\
\hline
1\ \vline&1.67e-4&-&1.54e-4&-&2.61e-4&-&1.08e-4&-\\
2\ \vline&4.19e-5&1.99&3.87e-5&1.99&6.54e-5&2.00&5.98e-5&0.85\\
3\ \vline&1.05e-5&2.00&9.69e-6&2.00&1.61e-5&2.02&2.47e-5&1.28\\
4\ \vline&2.62e-6&2.00&2.42e-6&2.00&3.98e-6&2.02&7.44e-6&1.73\\
5\ \vline&6.60e-7&1.99&6.09e-7&1.99&9.92e-7&2.00&1.92e-6&1.95\\
\hline
\end{tabular}
\end{center}
\footnotesize
\caption{GMRES numbers of iterations for Example \ref{ex:bdlayer} with different values of $\beta$ and $H/h=6$}\label{table:bdgmrescounts}
\begin{center}
\begin{tabular}{cccccc}
&\multicolumn{5}{c}{$k=1$}\\
\hline
\multirow{2}{*}{$\beta$}&\multicolumn{5}{c}{Number of subdomains}\\
\cline{2-6}\\[-2ex]
&$4^2$&$8^2$&$16^2$&$32^2$&$64^2$\\
\hline
$1$  &21  & 25  &25 &24 &24  \\
$10^{-2}$ &23 &25  &23   &21  &21 \\    
$10^{-3}$ &23  &24  &23   &21 &21  \\   
$10^{-4}$ &24  &24  &24   &23 &20 \\ 
$10^{-5}$ &23  &25  &23   &23 &21 \\ 
$10^{-6}$&15  &24  &23  &23 &22  \\ 
\hline
\end{tabular}
\end{center}
\end{table}

\section{Concluding Remarks}\label{sec:conclude}

In this work, we conduct a thorough analysis of the HDG methods for an optimal control problem constrained by a convection-diffusion-reaction equation. We proved the convergence in an energy norm and track the parameter $\beta$ explicitly. We also propose a BDDC algorithm to solve the discretized system and observe robustness with respect to $\beta$. The analysis framework of the HDG methods for optimal control problems can possibly be extended to convection-dominated state equations. It is also interesting to consider optimal control problems with pointwise state constraints (cf. \cite{brenner2023multigrid,brenner2021P1,liu2024discontinuous}). Meanwhile, the theoretical analysis of the convergence rates of the BDDC algorithms is under investigation in an ongoing work.

\section*{Acknowledgment}
The authors would like to thank Prof. Xuemin Tu for helpful discussions regarding this project.

\bibliographystyle{siamplain}
\bibliography{references}
\end{document}

%% file: ex_shared.tex
\usepackage{bm}
\usepackage{lipsum}
\usepackage{amsfonts,amsmath,amssymb}
\usepackage{graphicx}
\usepackage{epstopdf}
\usepackage{algorithmic}
\usepackage{tikz}
\usepackage{stmaryrd}
\usepackage{enumitem}
\usetikzlibrary{arrows.meta}
\usepackage{caption}
\usepackage{multirow}
\graphicspath{{figures/}}
\ifpdf
  \DeclareGraphicsExtensions{.eps,.pdf,.png,.jpg}
\else
  \DeclareGraphicsExtensions{.eps}
\fi

\newtheorem{example}[theorem]{Example}
\newtheorem{assumption}{Assumption}[section]

\numberwithin{equation}{section}

\def\bz{\bm{\zeta}}

\def\bn{\bm{n}}
\def\Ho{H^1_0(\Omega)}
\def\O{\Omega}

\def\LT{{L_2(\O)}}

\def\cE{\mathcal{E}}
\def\R{\mathcal{R}}
\def\cS{\mathcal{S}}
\def\cC{\mathcal{C}}
\def\cA{\mathcal{A}}

\def\ty{y}
\def\qb{\bm{q}}

\def\H2{{H^2(\O)}}

\def\ES{\Ho\times\Ho}

\def\cT{\mathcal{T}}
\def\cB{\mathcal{B}}

\def\A{\bm{A}}

\def\l{\langle}
\def\r{\rangle}

\def\g{\bm{g}}
\def\u{\bm{u}}
\def\L{\bm{L}}

\def\S{\bm{S}}

\def\D{\bm{D}}

\def\br{\bm{r}}
\def\bq{\bm{q}}
\def\bp{\bm{p}}
\def\lambdae{\bm\lambda}
\def\Lambdae{\bm\Lambda}

\DeclareMathOperator*{\argmin}{argmin}

\newsiamremark{remark}{Remark}
\newsiamremark{hypothesis}{Hypothesis}
\crefname{hypothesis}{Hypothesis}{Hypotheses}
\newsiamthm{claim}{Claim}

\headers{OCP-HDG-BDDC}{S. Liu and J. Zhang}

\title{A balancing domain decomposition by constraints preconditioner for a hybridizable discontinuous Galerkin discretization of an elliptic optimal control problem\thanks{Submitted to the editors 04/02/2025.
\funding{This material is based upon work supported by the National Science Foundation under Grant No. DMS-1929284 while S. L. was in residence at the Institute for Computational and Experimental Research in Mathematics in Providence, RI, during the "Numerical PDEs: Analysis, Algorithms, and Data Challenges" program.}}}

\author{Sijing Liu\thanks{Department of Mathematical Sciences, Worcester Polytechnic Institute, Worcester, MA
  (\email{sliu13@wpi.edu}).}
\and Jinjin Zhang\thanks{Department of Mathematics, The Ohio State University, Columbus, OH
  (\email{zhang.14647@osu.edu}).}}

\usepackage{amsopn}
